\pdfoutput=1
\documentclass[12pt]{amsart}
\usepackage[usenames,dvipsnames]{xcolor}
\usepackage{amsmath,amssymb,amsthm,mathtools,tikz,color,xspace,float,stmaryrd}
\usepackage{array}
\usepackage[letterpaper,margin=1in]{geometry}
\usetikzlibrary{decorations.pathreplacing,calligraphy}
\usetikzlibrary{shapes,arrows,svg.path}

\usepackage[font=footnotesize,labelfont=bf, width=.9\textwidth]{caption}

\tikzstyle{dot}=[fill=black, circle, inner sep=0pt, minimum size=3pt]
\tikzstyle{state}=[fill=white, draw, solid, circle, inner sep=0pt, minimum size=12pt]

\usepackage{hyperref}
\usepackage[protrusion=true,expansion=true]{microtype}
\usepackage{mathscinet} 

\usepackage[capitalize, noabbrev]{cleveref} 

\DeclareRobustCommand{\SkipTocEntry}[5]{}

\newtheorem{theorem}{\textbf{Theorem}}[section]
\newtheorem{proposition}[theorem]{\textbf{Proposition}}

\newtheorem{definition}[theorem]{Definition}
\newtheorem{remark}[theorem]{Remark}

\newtheoremstyle{plainnonitalic} 
{}                
{}                
{\normalfont}     
{}                
{\bfseries}       
{.}               
{5pt plus 1pt minus 1pt} 
{}                

\theoremstyle{plainnonitalic}

\DeclareMathOperator{\GL}{GL}

\DeclareMathOperator{\id}{1}
\newcommand{\abs}[1]{\lvert#1\rvert}

\hypersetup{%
    pdftitle={},
    pdfauthor={},
    linktocpage=true,           
    colorlinks=true,            
    linkcolor=blue!75!black,    
    citecolor=green!50!black,   
    filecolor=magenta,          
    urlcolor=blue!75!black      
}

\newcounter{Todo}

\begin{document}

\title{Quantum Superalgebras and the Free-Fermionic Yang-Baxter equation}

\author{Ben Brubaker}
\address{School of Mathematics, University of Minnesota, Minneapolis, MN 55455}
\email{brubaker@math.umn.edu}
\author{Daniel Bump}
\address{Department of Mathematics, Stanford University, Stanford, CA 94305-2125}
\email{bump@math.stanford.edu}
\author{Henrik P. A. Gustafsson}
\address{Department of Mathematics and Mathematical Statistics, Umeå University, SE-901 87 Umeå, Sweden}
\email{henrik.gustafsson@umu.se}

\begin{abstract}
The free-fermion point refers to a $\operatorname{GL}(2)\times\operatorname{GL}(1)$ parametrized
Yang-Baxter equation within the six-vertex model. It has been known for a long time that this
is connected with the quantum group $U_q(\mathfrak{gl}(1|1))$. 
We demonstrate that $R$-matrices from the finite quantum superalgebra $U_q(\mathfrak{gl}(1|1))$ 
produce a dense subset of the free-fermionic Yang-Baxter equations of the six-vertex model, matching those of the prime, simple modules
in the affine quantum superalgebra $U_q(\widehat{\mathfrak{gl}}(1|1))$. Either of these quantum
groups can be used to generate the full free-fermion point, and we discuss them both.
Our discussion includes 6 families of six-vertex models used by Brubaker, Bump,
and Friedberg in connection with Tokuyama's theorem, a deformation of the
Weyl character formula. Thus our work gives quantum group interpretations for those models, 
known informally as Tokuyama ice.
\end{abstract}


\maketitle

\section{Introduction}

The {\it free-fermion point} of the six-vertex model is a subset of possible Boltzmann weight choices
for which the model behaves like non-interacting fermions. Algebraically, the six Boltzmann weights
must satisfy a simple homogeneous quadratic equation \cite{Felderhof1, FanWuPhase, Baxter}. This condition guarantees that free-fermion
Boltzmann weights may be encoded in an $R$-matrix satisfying the Yang-Baxter equation. As shown on p. 126 of
\cite{KBI93} and in~\cite{BrubakerBumpFriedbergHkice}, the set of all free-fermionic weights give
an unusual parametrized Yang-Baxter equation with parameter group
$\operatorname{GL} (2) \times \operatorname{GL} (1)$. These Yang-Baxter equations ensure that the
partition functions of the associated lattice models are exactly solvable, with interesting connections to
special functions in algebra, combinatorics, and probability.

Another important source of $R$-matrices for the Yang-Baxter equation is pairs of modules of quantum groups, particularly
the quantized enveloping algebras $U_q(\mathfrak{g})$ for Lie algebras $\mathfrak{g}$. Indeed a primary motivation for
the genesis of quantum groups was to provide examples of Yang-Baxter equations \cite{Drinfeld}. It has long been understood (e.g., \cite{BazhanovShadrikov, Perk-Au-Yang, ABPW}) that
taking two irreducible modules for the superalgebra $U_q(\mathfrak{sl}(1|1))$ results in a six-vertex
model $R$-matrix at the free-fermion point. In contrast with $U_q(\mathfrak{sl}_2)$, the standard module is just one of
a continuous family of irreducible two-dimensional representations of $U_q(\mathfrak{sl}(1|1))$.
Now given any such $R$-matrix there are methods for developing
multi-parameter generalizations of Yang-Baxter equations, which have collectively become known as ``baxterization'' 
\cite{LiguoriMintchev, BoukraaMaillard, MaityPadmanabhanKorepin} after early work in this
direction by Vaughan Jones \cite{Jones} in the context of knot theory.
Thus it becomes natural to ask whether baxterization methods
exist to exhaust all free-fermionic solutions and whether these methods admit some natural interpretation in the context
of the affine quantum superalgebra $U_q(\widehat{\mathfrak{gl}}(1|1))$.

Bracken, Gould, Zhang, and Delius \cite{BrackenGouldZhangDelius} complete such a 
baxterization for several low rank groups including $U_q(\mathfrak{gl}(1|1))$, 
in which the highest weights of the respective modules feature in the additional parameters.
There is an additional spectral parameter via methods from Jimbo \cite{JimboQuantumGroups}.
We adopt a similar strategy in Section~\ref{sec:rmatrixderived}, beginning from two highest 
weight representations $V(\zeta)$ and $V(\zeta')$, and providing a rigorous derivation of two initial 
$R$-matrices which span $\textrm{End}(V(\zeta) \otimes V(\zeta'))$. 
We then show (Theorem~\ref{thm:YBE2}) parametrized Yang-Baxter equations for linear
combinations of the two initial $R$-matrices. The linear combinations of two $R$-matrices with two highest weights
give us four total parameters, short of the five parameter family at the free-fermion point and we
also identify the additional constraint equation to pin down our subset in Remark~\ref{rem:cubic}.
If we further allow for changes of basis in our modules, we gain one additional parameter and thus
achieve a dense subset of the free-fermion point. We hope to use a similar version of baxterization for
higher rank superalgebra modules in future work. 

Several things stand out. Most importantly, $U_q(\mathfrak{gl}(1|1))$, or
equivalently $U_q(\mathfrak{sl}(1|1))$, has a continuous family of irreducible two dimensional modules. 
In the classical case $\mathfrak{gl}(1|1)$, these are the Kac modules. 
This large family of two dimensional representations is in contrast to
$U_q(\mathfrak{sl}_2)$, and makes a decisive difference between the two
theories.
A further distinction is that the superalgebra has a central element in its Cartan subalgebra which we will denote by $W$, and the $W$-eigenvalues appear in the formulas.
On the other hand, the deformation parameter $q$ does not play an important role: all values
of $q$ give rise to the free-fermionic point, and $q$ disappears from
the formulas.

Another source of parametrized Yang-Baxter equations at the free-fermion point are modules
for quantum affine superalgebras. Indeed the spectral parameter in the $R$-matrix naturally 
arises from the central extension. The familiar complication here is that the computations for $R$-matrices
are rather intricate and depend heavily on the chosen presentation for the affine group. 
The Fadeev-Reshetikhin-Takhtajan style presentation \cite{FRT1, FRT2} immediately yields solutions
to Yang-Baxter equations by construction, but the representation theory may be more naturally described
in terms of the alternate presentations of Drinfeld and of Drinfeld-Jimbo. Moving amongst these different
presentations can be a challenge. Fortunately, we have explicit computations of $R$-matrices
for a large class of irreducible representations of $U_q(\widehat{\mathfrak{gl}}(1|1))$ due to 
Zhang \cite{HuafengZhangGL11}. Zhang shows in~\cite{HuafengZhangRTT} that this family is 
essentially all prime, simple modules for the quantum superalgebra. Here ``prime'' means it does not
arise as a non-trivial tensor product of modules. We review this in more detail in Section~\ref{sec:affine} and further
demonstrate that our earlier baxterization allows us to recover this class of modules. In this way,
we show that our more primitive methods for finite Lie algebras recover the affine $R$-matrices in question.

We analyze an important special class of free-fermionic
models that arose in \cite{BrubakerBumpFriedbergHkice}, which we refer to as
\textit{Tokuyama models}. Tokuyama~\cite{Tokuyama} described a deformation
of the Weyl character formula for $\operatorname{GL}(n)$, and his formula
can be interpreted as evaluating the partition function of a model of
this type. Actually there are two versions of Tokuyama's theorem,
meaning there are two different ways of expressing his formula as a sum over
Gelfand-Tsetlin patterns. Correspondingly there are two versions
of Tokuyama ice, that we call Gamma and Delta. They are
not interchangeable: for example it is possible to construct systems with Gamma and Delta
ice whose partition functions are both equal to Tokuyama's formula, but there is not a simple bijection of
states that implies this.
The equality of partition functions can, however, be proved using the Yang-Baxter equation.
In addition to Gamma Ice and Delta Ice there are four other flavors of free-fermionic vertices 
needed to complete the theory.

Gamma and Delta ice appear in Ivanov~\cite{Ivanov} to
construct U-turn models involving both sets of weights simultaneously, whose partition functions
are characters of the symplectic group. They appear in~\cite{BrubakerSchultzHamiltonians,BBBGVertex}
in connection with models whose row transfer matrices are vertex operators acting on
the fermionic Fock space. As explained in Chapters 1, 6 and~19 of~\cite{wmd5book} and
the Appendix in \cite{BBB}, generalizations of these Yang-Baxter equations for Tokuyama ice imply
the analytic continuation and functional equations of certain Type~A ``Weyl group
multiple Dirichlet series.''

In all of the above work, the permutability of rows of Gamma and Delta ice,
encoded in the four auxiliary families, plays a special role.

In the crystal limit $q=0$ Tokuyama models or their colored variants
specialize to five-vertex models whose partition
functions are Demazure atoms~\cite{BBBGdemice}; after a transformation these
limiting models include stochastic models also related to random walks
with osculating and vicious walkers~\cite{Fisher, ForresterDW, BrakOwczarek, KorffQC}. 

We review Gamma and Delta ice in Section~\ref{sec:tokuyama} and show how each of the aforementioned sets of Boltzmann weights for
Tokuyama ice are realized by baxterization (Theorem~\ref{thm:alltheice}).
From either the viewpoint of $U_q({\mathfrak{gl}}(1|1))$ or its affinization, Gamma and
Delta ice are limiting cases, while the other four families are not.

In Section~\ref{sec:affine}, we demonstrate how limits in the spectral parameter of the affine module $R$-matrices
preserve solutions to the Yang-Baxter equation and recover the weights for both Gamma and Delta models.
Algebraically, these limits of $R$-matrices should arise from limits of modules, known as ``asymptotic'' modules
in the literature and studied by Hernandez-Jimbo \cite{JimboHernandez} in the non-super setting and by 
Huafeng Zhang \cite{ZhangSuperAsymptotic} in the superalgebra case. Making the connection rigorous between limits of
$R$-matrices and asymptotic modules is an interesting avenue for future work.

Taken together, we thus provide two quantum group explanations for the existence of Gamma and Delta ice -- one
via baxterization of $R$-matrices for quantum groups of the finite Lie algebra and another as limits of affine quantum
group $R$-matrices. These explanations give resolutions of the long-standing mystery of the algebraic origins of Tokuyama ice, set in motion in \cite{BrubakerBumpFriedbergHkice}.
Ice-type models vastly generalizing Tokuyama ice and featuring colored paths moving in multiple directions were given in \cite{BBBGMetahori}. 
For these models, the transfer matrices satisfied Yang-Baxter equations coming from standard modules for quantum groups associated to
$\mathfrak{gl}(m|n)$ for arbitrary $m,n$. We hope to uncover the mysterious origins of the Boltzmann weights used in the row
transfer matrices by similar baxterization procedure in higher rank. Indeed this was our primary motivation for exploring the cases
in the present paper.

The literature on the free-fermion point of the six-vertex model and its generalizations is so vast that we have necessarily had to
limit the references. In addition to those already mentioned, we describe other papers with connection to the present topics from
which we have derived particular insight. For a recent surveys of integrable six-vertex models, see \cite{ReshetikhinLectures, ZJHab}. 
The recent work of~\cite{MaityPadmanabhanKorepin} explores baxterization of the non-Hermitian cases of Hietarinta's \cite{Hietarinta} constant 
solutions to the Yang-Baxter
equation. The paper \cite{BumpMcNamaraNakasuji} generalizes
\cite{BrubakerBumpFriedbergHkice} to factorial Schur functions, while \cite{ABPW} explores rational
functions generalizing factorial Schur functions with applications, for
example, to domino tilings.
\addtocontents{toc}{\SkipTocEntry}
\subsection*{Acknowledgements}
This paper would not have been possible without the many discussions and ideas
provided by Valentin Buciumas, particularly with regard to the limiting
processes used to obtain Gamma and Delta ice from affine $R$-matrices. We thank
Amol Aggarwal for first making us aware of the asymptotic modules and their
potential connection to the lattice models discussed here. We also thank
Hiroyuki Yamane for sharing some of his papers with us and Sasha Tsymbaliuk for
helpful conversations.
We thank Vladimir Korepin, Jacques Perk, Carl Westerlund and Huafeng Zhang for helpful feedback on an earlier version of this paper.
This work was partially supported by NSF grant DMS-2401470 (Brubaker).

\section{Lie Superalgebras}

We will quickly review some relevant theory referring to
{\cite{KacSuperalgebras,MussonSuperalgebras,ChengWangDualities}} for further
information.

A \textit{super vector space} $V$ is a $\mathbb{Z}_2$-graded vector space.
Elements of the $0$-graded part $V_0$ are called \textit{even} and elements
of $V_1$ are called odd. If $x \in V$ is homogeneous, we will denote by $|x|
\in \{0, 1\}$ the graded degree of $x \in V$. Let $\mathbb{C}^{m|n}$ be
the super vector space with $(\mathbb{C}^{m|n})_0 =\mathbb{C}^m$ and
$(\mathbb{C}^{m|n})_1 =\mathbb{C}^n$.

If $V$ and $W$ are super vector spaces, then so are $V \oplus W$ and $V
\otimes W$ with the gradings
\[ (V \oplus W)_0 = V_0 \oplus W_0, \qquad (V \oplus W)_1 = V_1 \oplus W_1, \]
\[ (V \otimes W)_0 = V_0 \otimes W_0 \oplus V_1 \otimes W_1, \qquad (V \otimes
   W)_1 = V_0 \otimes W_1 \oplus V_1 \otimes W_0 . \]
Furthermore $\mathrm{Hom} (V, W)$ is a super vector space, where $f : V
\longrightarrow W$ is even if $f (V_0) \subseteq W_0$ and $f (V_1) \subseteq
W_1$, and odd if $f (V_0) \subseteq W_1$ and $f (V_1) \subseteq W_0$.

A \textit{superalgebra} (associative or not) is a $\mathbb{Z}_2$-graded
algebra. The grading thus satisfies $|ab| = |a| + |b|$ in $\mathbb{Z}_2$. Thus
it is a super vector space with a compatible multiplicative structure.

If $A, B$ is are superalgebras, then $A \otimes B$ is also an associative
superalgebra, with the graded multiplication defined for homogeneous elements
by
\begin{equation}
  \label{eq:shopfsign} (a \otimes b)  (c \otimes d) = (- 1)^{|b| \cdot |c|} 
  (ac \otimes bd) .
\end{equation}
Also if $V, W$ are modules for $A$ and $B$, we must use the graded tensor
product in making $V \otimes W$ into a module for $A \otimes B$. Thus if $a
\otimes b \in A \otimes B$ and $v \otimes w \in V \otimes W$ are homogeneous
\begin{equation}
  \label{eq:modulesign} (a \otimes b)  (v \otimes w) = (- 1)^{|b| \cdot |v|} 
  (av \otimes bw) .
\end{equation}
A \textit{Lie superalgebra} is a super vector space $\mathfrak{g}$ with a
graded nonassociative algebra with the ``multiplication'' denoted $\left[
\hspace{0.27em}, \hspace{0.27em} \right]$, subject to the axioms (for
homogeneous $a, b, c$)
\begin{gather*}
[a, b] = - (- 1)^{|a| |b|} [b, a], \\
(- 1)^{|a| |c|} [a, [b, c]] + (- 1)^{|a| |b|} [b, [c, a]] + (- 1)^{|b| |c|}
   [c, [a, b]] = 0.
\end{gather*}
As an example, if $A$ is an associative superalgebra then we may make the
underlying vector space of $A$ into a Lie superalgebra by defining the bracket
\[ [a, b] = ab - (- 1)^{|a| |b|} ba. \]
Taking $V$ to be a super vector space, $\mathrm{End} (V)$ with this super Lie
algebra is denoted $\mathfrak{g}\mathfrak{l} (V)$. Alternatively if $V
=\mathbb{C}^{m|n}$ we denote it $\mathfrak{g}\mathfrak{l} (m|n)$. Write
\[ \mathrm{End} (V) = \mathrm{End} (V_0) \oplus \mathrm{Hom} (V_0, V_1) \oplus
   \mathrm{Hom} (V_1, V_0) \oplus \mathrm{End} (V_1). \]
If $f = (\alpha, \beta, \gamma, \delta) \in \mathrm{End} (V)$ under this
decomposition the \textit{supertrace} is 
\[\mathrm{str} (f) := \mathrm{tr} (\alpha) - \mathrm{tr} (\delta).\]
Its kernel $\mathfrak{s}\mathfrak{l} (V)$ is a subalgebra of $\mathfrak{g}\mathfrak{l}
(V)$. If $V =\mathbb{C}^{m|n}$ then we denote this $\mathfrak{s}\mathfrak{l}
(m|n)$.

\section{The Lie superalgebra \texorpdfstring{$\mathfrak{g}\mathfrak{l} (1|1)$}{gl(1|1)}\label{sec:gl11}}

The representation theory of Lie superalgebras such as
$\mathfrak{g}\mathfrak{l} (m|n)$ is both similar and different from that of
complex reductive Lie algebras such as $\mathfrak{g}\mathfrak{l} (N)$.
Similarities include the parametrization of irreducibles by highest weights,
the Category~$\mathcal{O}$ theory, and the existence of quantized enveloping
algebras, leading to solutions of the Yang-Baxter equation. A significant
difference is the fact that finite-dimensional representations may not be
completely reducible. In the superalgebra case there is a class of Verma
modules -- the Kac modules -- that are finite-dimensional.

The Lie superalgebra $\mathfrak{g} = \mathfrak{gl}(1|1)$ has a basis $E, F, Z, H$ which will be detailed below, while
$\mathfrak{g}' = \mathfrak{sl}(1|1)$ is spanned by $E, F, Z$. The Lie superalgebra $\mathfrak{g}$
has a faithful 2-dimensional \textit{standard module} which is a
$(1|1)$-dimensional vector superspace, on which the basis elements of
$\mathfrak{g}$ are represented by the matrices
\[ E = \left( \begin{array}{cc}
     0 & 1\\
     0 & 0
   \end{array} \right), \qquad F = \left( \begin{array}{cc}
     0 & 0\\
     1 & 0
   \end{array} \right), \qquad Z = \left( \begin{array}{cc}
     1 & 0\\
     0 & 1
   \end{array} \right), \qquad H = \left( \begin{array}{cc}
     1 & 0\\
     0 & - 1
   \end{array} \right) . \]
Here $Z, H$ span the even part $\mathfrak{g}_0$ of $\mathfrak{g}$, and $E, F$
span the odd part $\mathfrak{g}_1$. The even part $\mathfrak{g}_0$ coincides
with the two-dimensional Cartan subalgebra $\mathfrak{h}$. We will also denote by
$\mathfrak{h}'$ the one-dimensional Cartan subalgebra of $\mathfrak{g}'$ spanned 
by $Z$. The Lie superalgebra bracket operation satisfies
\[ [E, F] = Z, \qquad [E, E] = [F, F] = 0, \qquad [H, E] = 2 E, \qquad [H, F] = -2 F \]
and $[Z, X] = 0$ for all $X \in \mathfrak{g}$.
Let $U (\mathfrak{g})$ and $U (\mathfrak{g}')$ be the universal enveloping
algebras, generated by the basis vectors subject to the relation
\[ [X, Y] = XY - (- 1)^{|X| |Y|} YX. \]
Now $U (\mathfrak{g}')$ has center $\mathbb{C} [Z]$ and the other generators
satisfy $EF + FE = Z$, while $E^2 = F^2 = 0$ and so as an abstract ring, it
resembles a Clifford algebra. More precisely, if $\zeta$ is any nonzero
constant then quotienting by the ideal generated by $Z - \zeta$ gives the
four-dimensional Clifford algebra of a two-dimensional quadratic space.

We will see now that the nontrivial irreducible modules of $U (\mathfrak{g})$
or $U (\mathfrak{g}')$ are all two dimensional, and are Kac modules. For
definiteness we will deal with $U (\mathfrak{g})$ but the other case is almost
the same.

Let $\mathfrak{n}_+ =\mathbb{C}E$ and $\mathfrak{n}_- =\mathbb{C}F$. Now let
$\lambda : \mathfrak{h} \longrightarrow \mathbb{C}$ be any linear functional.
We extend this to the Borel subalgebra $\mathfrak{b}=\mathfrak{h} \oplus
\mathfrak{n}_+$ by letting $\lambda (E) = 0$. This is the character of a
one-dimensional module $\mathbb{C}_{\lambda}$ of $\mathfrak{b}$. A
\textit{highest weight module $V$} of weight $\lambda$ is a module generated
by a single vector $v_{\lambda}$ such that $bv_{\lambda} = \lambda (b)
v_{\lambda}$ for $b \in \mathfrak{b}$. There is a universal highest weight
module $K (\lambda) = U (\mathfrak{g}) \otimes_{U (\mathfrak{b})}
\mathbb{C}_{\lambda}$, called the \textit{Kac module}.

By the Poincar{\'e}-Birkhoff-Witt theorem $U (\mathfrak{g}) \cong U
(\mathfrak{n}_-) \otimes U (\mathfrak{b})$ and so $K (\lambda) \cong U
(\mathfrak{n}_-)$ as a vector space via the map $\xi \mapsto \xi \otimes
v_{\lambda}$ from $U (\mathfrak{n}_-)$ to $K (\lambda)$. But $U
(\mathfrak{n}_-)$ is odd and abelian and so $U (\mathfrak{n}_-) \cong
\bigwedge \mathfrak{n}_-$, the superalgebra. Because $\dim (\mathfrak{n}_-) =
1$ the module $K (\lambda)$ is 2-dimensional. For this
$\mathfrak{g}=\mathfrak{g}\mathfrak{l} (1|1)$, the module $K (\lambda)$ is
irreducible unless $\lambda = 0$.

Now every irreducible module $V$ is generated by a highest weight vector for
some $\lambda \in \mathfrak{h}^{\ast}$, so $V$ is a quotient of the universal
highest weight module $K (\lambda)$. This implies for
$\mathfrak{g}=\mathfrak{g}\mathfrak{l} (1|1)$ that every irreducible module is
either trivial (if $\lambda = 0$) or is itself a Kac module. This is atypical
since the story for general $\mathfrak{g}\mathfrak{l} (m|n)$ is quite different.

\section{The quantized enveloping algebra \texorpdfstring{$U_q(\mathfrak{g})$}{Uq(g)}\label{sec:rmatrixderived}}

Given an associative superalgebra, with graded homomorphisms $\Delta :
\mathcal{H} \longrightarrow \mathcal{H} \otimes \mathcal{H}$ and $\eta :
\mathcal{H} \longrightarrow \mathbb{C}$ to serve as comultiplication and
counit, $\mathcal{H}$ becomes a Hopf superalgebra if the usual Hopf algebras
axioms are satisfied. It is important to pay attention to sign conventions for
tensor products, which we review.

\

With $\mathfrak{g}=\mathfrak{g}\mathfrak{l} (1|1)$ and $\mathfrak{g}'
=\mathfrak{s}\mathfrak{l} (1|1)$, the quantized enveloping algebras $U_q
(\mathfrak{g})$ and $U_q (\mathfrak{g}')$ are Hopf superalgebras.

Let $q$ be a parameter which may be an indeterminate or complex number such
that $q^4 \neq 1$. If $\zeta \in \mathbb{C}^{\times}$ we will frequently
encounter the quantity
\[ [\zeta] = \frac{\zeta - \zeta^{- 1}}{q - q^{- 1}} . \]
Let $E, F, Z$ and $H$ be as in the last section. Now let (formally) $W = q^Z$
and $K = q^H$. Then
\[ U_q (\mathfrak{g}) =\mathbb{C} [E, F, W, W^{- 1}, K, K^{- 1}], \qquad U_q
   (\mathfrak{g}') =\mathbb{C} [E, F, W, W^{- 1}] . \]
We will now describe these super Hopf algebras, following
Kwon~{\cite{KwonKac}}. We deal mainly with~$U_q(\mathfrak{g})$. Although $K$
plays a much less important role than $W$ in the calculation, at one point we
will make use of $K$.

The generators are now subject to the relations
\begin{equation}
  \label{eq:genrel1} KEK^{- 1} = q^2 E, \qquad KFK^{- 1} = q^{- 2} F,
\end{equation}
and
\begin{equation}
  \label{eq:genrel2} WEW^{- 1} = E, \qquad WFW^{- 1} = F, \qquad EF + FE =
  \frac{W - W^{- 1}}{q - q^{- 1}} .
\end{equation}
In particular, $W$ is central. We have also
\[ E^2 = F^2 = 0 \]
in the quantized enveloping algebra. The comultiplication is a homomorphism
$\Delta : \mathcal{H} \longrightarrow \mathcal{H} \otimes \mathcal{H}$ of
graded superalgebras. On the basis vectors,
\[ \Delta (W) = W \otimes W, \qquad \Delta (W^{- 1}) = W^{- 1} \otimes W^{-
   1}, \qquad \Delta (K) = K \otimes K, \qquad \Delta (K^{- 1}) = K^{- 1}
   \otimes K^{- 1}, \]
\[ \Delta (E) = E \otimes W^{- 1} + 1 \otimes E, \qquad \Delta (F) = F \otimes
   1 + W \otimes F. \]
It may be checked that $\Delta (E)$, $\Delta (F)$, $\Delta (W)$ and $\Delta
(K)$ satisfy the same relations (\ref{eq:genrel1}) and (\ref{eq:genrel2}) as
$E, F, W, K$, so this homomorphism exists. Note that checking $\Delta (E)
\Delta (F) + \Delta (F) \Delta (E)$ equals $\frac{\Delta (W) - \Delta (W)^{-
1}}{q - q^{- 1}}$ requires a cancellation that depends on the sign in
\eqref{eq:shopfsign}. The antipode is given by
\[ S (W) = W^{- 1}, \qquad S (Z) = Z^{- 1}, \qquad S (E) = - EW, \qquad S (F)
   = - W^{- 1} F. \]
The counit homomorphism $\varepsilon : \mathcal{H} \longrightarrow \mathbb{C}$
is defined by $\varepsilon (W) = \varepsilon (Z) = 1$, $\varepsilon (E) =
\varepsilon (F) = 0$.

The Hopf subalgebra $U_q (\mathfrak{g}')$ of $U_q (\mathfrak{g})$ omits the
generator $K$. Note that this is closed under both multiplication and
comultiplication. We will denote by $\mathfrak{h}$ the subalgebra of $U_q
(\mathfrak{g})$ generated by $W$ and $Z$, and by $\mathfrak{h}'$ the
subalgebra of $U_q (\mathfrak{g}')$ generated by $W$.

Given eigenvalues $\zeta$ and $\kappa$ of $W$ and $K$, respectively, there is
a universal highest weight module (i.e.\ Kac module) $V (\zeta, \kappa)$ of $U_q
(\mathfrak{g})$ that is always 2-dimensional and irreducible unless $\zeta =
\pm 1$. Similarly we will denote by $V (\zeta)$ the universal highest weight
module of $U_q (\mathfrak{g}')$ that is the universal highest weight module
for the eigenvalue $\zeta$ of $W$.

We begin by analyzing $V (\zeta, \kappa)$. The eigenvalue $\kappa$ will play
little role in the results.

\begin{proposition}
  \label{prop:basis}
  Let $\kappa, \zeta \in \mathbb{C}^{\times}$. 
  Then $U_q (\mathfrak{g})$ has a $(1|1)$-dimensional module $V (\zeta,
  \kappa)$ with basis $\{x, y\}$, where $x$ is even and $y$ is odd, and
  \begin{equation}
    \label{eq:erelations} Ex = 0, \qquad Fx = y, \qquad Wx = \zeta x, \qquad
    Kx = \kappa x,
  \end{equation}
  \begin{equation}
    \label{eq:frelations} Fy = 0, \qquad Ey = [\zeta] x, \qquad Wy = \zeta y,
    \qquad Ky = q^{- 2} \kappa y.
  \end{equation}
  Any module generated by a vector $x$ satisfying \eqref{eq:erelations} must
  also satisfy \eqref{eq:frelations} hence is a quotient of $V (\zeta,
  \kappa)$. But $V (\zeta, \kappa)$ is irreducible unless $\zeta = \pm 1$.
\end{proposition}

\begin{proof}
  We leave it to the reader to check that these relations define a module.
  However we will show how the relations (\ref{eq:erelations}) imply
  (\ref{eq:frelations}); this derivation contains most of the checking that
  needs to be done to see that $V (\zeta, \kappa)$ is a module. Indeed, since
  $F^2 = 0$ we have $Fy = 0$. We have
  \[ Ey = EFx = FEx + \left( \frac{W - W^{- 1}}{q - q^{- 1}} \right) x =
     \left( \frac{W - W^{- 1}}{q - q^{- 1}} \right) x = [\zeta] x. \]
  Because $W$ is central and $Wx = \zeta x$, we must have $Wy = \zeta y$. And
  \[ Ky = KFx = q^{- 2} FKx = q^{- 2} \kappa y. \]
\end{proof}

\begin{proposition}
  \label{prop:endozz} Let $\zeta, \kappa, \zeta', \kappa'$ be nonzero complex numbers
  such that $\zeta \neq \pm 1$, $\kappa^4 \neq 1$, with the same assumptions
  on $\zeta'$ and $\kappa'$. Suppose that
  \[ \phi : V (\zeta, \kappa) \otimes V (\zeta', \kappa') \longrightarrow V
     (\zeta', \kappa') \otimes V (\zeta, \kappa) \]
  is a module isomorphism. Let $x, y$ and $x', y'$ be the standard bases of $V
  (\zeta, \kappa)$ and $V (\zeta', \kappa')$. Then $\phi$ satisfies
  \begin{equation}
    \label{eq:phiabcd} \begin{array}{l}
      \phi (x \otimes x') = a_2 x' \otimes x,\\
      \phi (x \otimes y') = c_2 x' \otimes y + b_2 y' \otimes x,\\
      \phi (y \otimes x') = b_1 x' \otimes y + c_1 y' \otimes x,\\
      \phi (y \otimes y') = a_1 y' \otimes y
    \end{array}
  \end{equation}
  where $a_1$, $a_2$, $b_1$, $b_2$, $c_1$ and $c_2$ are constants.
  
  Furthermore, given $\phi$ satisfying \eqref{eq:phiabcd} a necessary and sufficient
  condition for $\phi$ to be a module homomorphism is that the following
  identities are satisfied.
  \begin{equation}
    \label{eq:phiconditions} \begin{aligned}
      {}[\zeta] c_2 + [\zeta'] \zeta^{- 1} b_2 &= [\zeta'] a_2, & [\zeta]
      b_1 + \zeta^{- 1} [\zeta'] c_1 &= [\zeta] (\zeta')^{- 1} a_2,\\
      {}[\zeta'] \zeta^{- 1} a_1 &= (\zeta')^{- 1} [\zeta] c_2 - [\zeta'] b_1, &
      - [\zeta] a_1 &= (\zeta')^{- 1} [\zeta] b_2 - [\zeta'] c_1,\\
      a_2 &= c_1 + \zeta b_2, & \zeta' a_2 &= b_1 + \zeta c_2, \\
      c_2 -
      \zeta' b_2 &= a_1, & b_1 - c_1 \zeta' &= - a_1 \zeta .
    \end{aligned}
  \end{equation}
\end{proposition}

\begin{proof}
  The $K$-eigenspaces in $V (\zeta, \kappa)$ are as follows.
  \[ \begin{array}{|l|l|l|l|}
       \hline
       \textbf{eigenvalue} & \kappa \kappa' & q^{- 2} \kappa \kappa' & q^{- 4}
       \kappa \kappa'\\
       \hline
       \textbf{eigenbasis}& x \otimes x' & \begin{array}{c}
         x \otimes y',\\
         y \otimes x'
       \end{array} & y \otimes y'\\
       \hline
     \end{array} \]
  These must be mapped by $\phi$ to the corresponding eigenspaces which proves the first statement.
  
  Now assume that $\phi$ has the form (\ref{eq:phiabcd}). For $\phi$ to be a
  module homomorphism there are several constraints, listed in the following
  table.
  \[ \def\arraystretch{1.25}\begin{array}{|l|l|}
       \hline
       \textbf{Equation} & \textbf{Conditions on parameters} \\
       \hline
       \hline
       \phi (E (x \otimes x')) = E (\phi (x \otimes x')) & \text{No conditions
       required}\\
       \hline
       \phi (E (x \otimes y')) = E (\phi (x \otimes y')) & [\zeta] c_2 +
       [\zeta'] \zeta^{- 1} b_2 = [\zeta'] a_2\\
       \hline
       \phi (E (y \otimes x')) = E (\phi (y \otimes x')) & [\zeta] b_1 +
       \zeta^{- 1} [\zeta'] c_1 = [\zeta] (\zeta')^{- 1} a_2\\
       \hline
       \phi (E (y \otimes y')) = E (\phi (y \otimes y')) & \begin{aligned}
         {}[\zeta'] \zeta^{- 1} a_1 &= (\zeta')^{- 1} [\zeta] c_2 - [\zeta']
         b_1\\
         - [\zeta] a_1 &= (\zeta')^{- 1} [\zeta] b_2 - [\zeta'] c_1,
       \end{aligned}\\
       \hline
       \phi (F (x \otimes x')) = F (\phi (x \otimes x')) & \begin{aligned}
         a_2 &= c_1 + \zeta b_2\\
         \zeta' a_2 &= b_1 + \zeta c_2
       \end{aligned}\\
       \hline
       \phi (F (x \otimes y')) = F (\phi (x \otimes y')) & c_2 - \zeta' b_2 =
       a_1\\
       \hline
       \phi (F (y \otimes x')) = F (\phi (y \otimes x')) & b_1 - c_1 \zeta' =
       - a_1 \zeta\\
       \hline
       \phi (F (y \otimes y')) = F (\phi (y \otimes y')) & \text{No conditions
       required}\\
       \hline
     \end{array} \]
  These are the conditions in (\ref{eq:phiconditions}). We will check a couple
  of these, leaving the rest to the reader. The first and last are trivial.
  Turning to the second, since $\Delta E = E \otimes W^{- 1} + 1 \otimes E$ we
  have
  \[ E (x \otimes y') = Ex \otimes W^{- 1} y' + x \otimes Ey' = x \otimes Ey'
     = [\zeta'] x \otimes x' . \]
  Applying $\phi$ to this gives $a_2 [\zeta'] x' \otimes x$. On the other hand
  \[ E (\phi (x \otimes y')) = E (c_2 x' \otimes y + b_2 y' \otimes x) = c_2
     [\zeta] x' \otimes x + [\zeta'] \zeta^{- 1} b_2 x' \otimes x \]
  and comparing we get $a_2 [\zeta'] = [\zeta] c_2 + [\zeta'] \zeta^{- 1}
  b_2$.
  
  We skip to the fourth condition. We have
  \[ E (y \otimes y') = [\zeta] (\zeta')^{- 1} x \otimes y' - [\zeta'] y
     \otimes x' . \]
  The sign of the second term comes from the fact that $E$ and $y$ are both
  odd, and (\ref{eq:modulesign}). Applying $\phi$ gives
  \[ \phi (E (y \otimes y')) = [\zeta] (\zeta')^{- 1}  (c_2 x' \otimes y + b_2
     y' \otimes x) - [\zeta']  (b_1 x' \otimes y + c_1 y' \otimes x) . \]
  On the other hand
  \[ E (\phi (y \otimes y')) = a_1 E (y' \otimes y) = a_1 [\zeta'] \zeta^{- 1}
     x' \otimes y - a_1 [\zeta] y' \otimes x. \]
  Comparing coefficients gives two identities:
  \[ \zeta^{- 1} [\zeta'] a_1 = (\zeta')^{- 1} [\zeta] c_2 - [\zeta'] b_1,
     \qquad - [\zeta] a_1 = (\zeta')^{- 1} [\zeta] b_2 - [\zeta'] c_1 \]
  We leave the remaining cases to the reader.
\end{proof}

\begin{theorem}
  \label{thm:homtwodim}Let $\zeta, \kappa$ and $\zeta', \kappa'$ be nonzero
  complex numbers such that $\zeta \neq \pm 1$, $\kappa^4 \neq 1$, with the
  same assumptions on $\zeta'$ and $\kappa'$. Let $V (\zeta, \kappa)$ and
  $V (\zeta', \kappa')$ be the corresponding $U_q(\mathfrak{g})$ modules. Then the space of homomorphisms
  \[ V (\zeta , \kappa) \otimes V (\zeta', \kappa') \longrightarrow V
     (\zeta', \kappa') \otimes V (\zeta , \kappa) \]
  is two-dimensional, spanned by $R = R_{\zeta, \kappa, \zeta', \kappa'}$, $R'
  = R_{\zeta, \kappa, \zeta', \kappa'}'$, where in the notation
  \eqref{eq:phiabcd}, we have
  \[ \begin{array}{|l||l|l|l|l|l|l|}
       \hline
       & a_1 & a_2 & b_1 & b_2 & c_1 & c_2\\
       \hline \hline
       R & - \zeta \zeta' & 1 & \zeta' & \zeta & 1 - \zeta^2 & 0\\
       \hline
       R' & 1 & - \zeta \zeta' & - \zeta & - \zeta' & 0 & 1 - (\zeta')^2\\
       \hline
     \end{array} \]
\end{theorem}

\begin{proof}
  It may be checked that the two assignments of $a_1, a_2, b_1, b_2, c_1, c_2$
  satisfy the linear equations (\ref{eq:phiconditions}). In both cases two
  equations involve both $[\zeta]$ and $[\zeta']$, and checking these requires
  some algebra using $[\zeta] / [\zeta'] = (\zeta - \zeta^{- 1}) / (\zeta' -
  (\zeta')^{- 1})$.
  
  To see that these solutions span the set of solutions to
  (\ref{eq:phiconditions}), given any solution we may subtract a linear
  combination of $R$ and $R'$ to obtain another solution with $c_2 = c_1 = 0$.
  (This requires $\zeta, \zeta' \neq \pm 1$, which we have assumed.) But if
  $c_1 = c_2 = 0$ again using $\zeta, \zeta' \neq \pm 1$, it is easy to deduce
  from the equations (\ref{eq:phiconditions}) that $a_1 = a_2 = b_2 = b_1 =
  0$.
\end{proof}

\begin{remark}
  It is easy to check that
  \begin{equation}
    \label{eq:rtwoone} R' = - \zeta \zeta' R^{- 1} .
  \end{equation}
\end{remark}

\begin{remark}
  Let $V = V (\zeta, \kappa)$ and $U = V (\zeta', \kappa')$ for short. Note that in $\mathrm{Hom} (V \otimes U, U \otimes V)$, the matrices $R$ and
  $R'$ are even graded since they map $(V \otimes U)_0$ (spanned by $x \otimes
  x'$ and $y \otimes y'$) to $(U \otimes V)_0$, and $(V \otimes U)_1$ (spanned
  by $x \otimes y'$ and $y' \otimes x$) to $(U \otimes V)_1$.
\end{remark}

Since $R$ and $R'$ do not depend on $\kappa$ and $\kappa'$ we will use the notation $R_{\zeta,\zeta'}$ and $R'_{\zeta,\zeta'}$ for when we want to specify $\zeta$ and $\zeta'$.

\begin{remark}
  If $\zeta = \zeta'$, then it is easy to check the {\emph{skein relation}}:
  \begin{equation}
    \label{eq:skein} R_{\zeta, \zeta} + R'_{\zeta, \zeta} = (1 - \zeta^2) I_{V
    (\zeta,\kappa) \otimes V (\zeta,\kappa)}.
  \end{equation}
  Combining this with $R R' = - \zeta^2 I$ we see that $R$ and $R'$ are roots
  of the polynomial $x^2 + {(\zeta^2 - 1)x} - \zeta^2$. The eigenvalues of $R$
  and $R'$ are $1$ and $- \zeta^2$, each with multiplicity 2, since $R$ and
	$R'$ both have this characteristic polynomial on both $\bigl(V (\zeta)\otimes V(\zeta)\bigr)_0$ 
	(spanned by $x \otimes x$ and $y \otimes y$) and $\bigl(V(\zeta)\otimes V(\zeta)\bigr)_1$
  (spanned by $x \otimes y$ and $y \otimes x$). Then $R^2$ and $(R')^2$ have
  eigenvalues $1$ and $\zeta^4$.
  For later use we record the following equation
  \begin{equation}
    \label{eq:doubleskein} R_{\zeta, \zeta}^2 + (R'_{\zeta, \zeta})^2 =  (R_{\zeta,\zeta} + R'_{\zeta,\zeta})^2 - 2 R_{\zeta,\zeta} R'_{\zeta,\zeta} =
    (1 +
    \zeta^4) I_{V (\zeta,\kappa) \otimes V (\zeta,\kappa)}.
  \end{equation}
\end{remark}

\section{The quantized enveloping algebra $U_q(\mathfrak{\mathfrak{sl}}(1|1))$}

Recall that $\mathfrak{g}' = \mathfrak{sl}(1|1)$. In view of the fact that 
the eigenvalue $\kappa$ of $K$ does not appear in the formulas in
Theorem~\ref{thm:homtwodim} there is no role for $\kappa$ in the theory,
and so there is no advantage to using $\mathfrak{gl}(1|1)$ instead of
$\mathfrak{sl}(1|1)$. 

Therefore we will switch to $\mathfrak{g}'$, whose modules are $V(\zeta)$.
Now we study homomorphisms of modules of $U_q (\mathfrak{g}')$. We will
require that a homomorphism of modules preserves the grading. The Kac module
$V (\zeta)$ of $U_q (\mathfrak{g}')$ can be extended in many ways to a module
$V (\zeta, \kappa)$ by choosing an eigenvalue $\kappa$ of $K$. Any of the
homomorphisms $V (\zeta, \kappa) \otimes V (\zeta', \kappa') \longrightarrow V
(\zeta', \kappa') \circ V (z, \kappa)$ constructed in
Theorem~\ref{thm:homtwodim} are {\textit{a fortiori}} homomorphisms $V (\zeta)
\otimes V (\zeta') \longrightarrow V (\zeta') \otimes V (\zeta)$ over $U_q
(\mathfrak{g}')$.
One of our main results is that there are no others.

\begin{theorem}
  \label{thm:homtwodimbis}Let $\zeta, \zeta'$ be nonzero complex numbers.
  Assume that $\zeta, \zeta', \zeta\zeta'\neq \pm 1$. 
  Then the space of $U_q (\mathfrak{g}')$-module homomorphisms $V (\zeta)
  \otimes V (\zeta') \longrightarrow V (\zeta') \otimes V (\zeta)$ is two
  dimensional, spanned by $R$ and $R'$ from Theorem~\ref{thm:homtwodim}.
\end{theorem}

\begin{proof}
  We will denote by $x, y$ and $x', y'$ the graded bases of $V (\zeta)$ and $V
  (\zeta')$, as in the notation introduced in Proposition~\ref{prop:endozz}.
  Let $\phi : V (\zeta) \otimes V (\zeta') \longrightarrow V (\zeta') \otimes
  V (\zeta)$ be any endomorphism. Since $\phi$ preserves the grading, it has
  the form:
  \begin{equation}
    \label{eq:phiwithd} \begin{array}{l}
      \phi (x \otimes x') = a_2 x' \otimes x + d_1 y' \otimes y,\\
      \phi (x \otimes y') = c_2 x' \otimes y + b_2 y' \otimes x,\\
      \phi (y \otimes x') = b_1 x' \otimes y + c_1 y' \otimes x,\\
      \phi (y \otimes y') = d_2 x' \otimes x + a_1 y' \otimes y.
    \end{array}
  \end{equation}
  We will show that $d_1 = d_2 = 0$. If this is known, then the relations in
  Proposition~\ref{prop:endozz} are proved the same way, and the result
  follows as in the proof of Theorem~\ref{thm:homtwodim}.
  
  On our assumption that $\zeta \zeta' \neq \pm 1$, the homomorphisms $R$ and
  $R'$, restricted to the two-dimensional subspace spanned by $x \otimes x'$
  and $y \otimes y'$ are linearly independent, and so we may subtract a linear
  combination of them to arrange that $a_1 = a_2 = 0$. This subtraction does
  not affect $d_1$ or $d_2$, so without loss of generality $a_1 = a_2 = 0$.
  
  Since $E (x \otimes x') = 0$ we get from \eqref{eq:frelations} that
  \[ 0 = \phi (E (x \otimes x')) = E (\phi (x \otimes x')) = d_1 E (y' \otimes
     y) = d_1 (\zeta^{- 1} [\zeta'] x' \otimes y - [\zeta] x \otimes y') . \]
  On our assumption that $\zeta \neq \pm 1$ we have $[\zeta] \neq 0$ and so
  this implies that $d_1 = 0$. We prove that $d_2 = 0$ the same way
  starting with $F (y \otimes y') = 0$.
\end{proof}

\begin{proposition}
  Let $\phi : V (\zeta) \otimes V (\zeta') \longrightarrow V (\zeta') \otimes
  V (\zeta)$ be an intertwining operator. Let $a_1, a_2, b_1, b_2, c_1, c_2$
  be defined by \eqref{eq:phiabcd}. Then
  \begin{equation}
    \label{eq:freefermioniccondition} a_1 a_2 + b_1 b_2 = c_1 c_2 .
  \end{equation}
\end{proposition}

\begin{proof}
  By Theorem~\ref{thm:homtwodimbis}, $\phi$ is a linear combination $uR + vR'$
  for suitable constants $u$ and $v$. Therefore the values (\ref{eq:phiabcd})
  are
  \begin{equation}
    \label{eq:sixbw}
    \begin{aligned}
      a_1 &= v - u \zeta \zeta', \qquad & a_2 &= u - v \zeta \zeta', \\
      b_1 &= u \zeta' - v \zeta, & b_2 &= u \zeta - v \zeta', \\
      c_1 &= u (1 - \zeta^2), & c_2 &= v (1 - (\zeta')^2) . 
    \end{aligned}
  \end{equation}
  These can be checked to satisfy (\ref{eq:freefermioniccondition}).
\end{proof}

\begin{remark}\label{rem:cubic}
The condition (\ref{eq:freefermioniccondition}) is called the
\textit{free-fermionic condition}~{\cite{FanWuPhase}}. 
The space of solutions to this condition is five-dimensional, but
(\ref{eq:sixbw}) parametrizes only a four-dimensional space.
Hence these solutions do not quite exhaust the free-fermionic vertices.
Indeed, we find that the parameters in~\eqref{eq:sixbw} also satisfy the cubic relation
\begin{equation}
  b_2^3 + a_1 b_1 c_1 + a_2 b_2 c_1 + a_2 b_1 c_2 + a_1 b_2 c_2 = 2 a_1 a_2
   b_1 + \left( a_1^2 {+ a_2^2}  + b_1^2 \right) b_2. 
\end{equation}
This may seem to mean that the solutions to the Yang-Baxter equation 
that we can obtain by this method are a proper subset of all free-fermionic
Yang-Baxter equations. However if we further enlarge the set 
of Yang-Baxter equations in a trivial way by a change of
basis in the modules $V(\zeta,\kappa)$, we obtain a dense subset of
all free-fermionic Yang-Baxter equations.
\end{remark}

\section{Yang-Baxter equations}

Recall that $V(\zeta,\kappa)$ is a module for $U_q(\mathfrak{g})$,
while $V(\zeta)$ is a module for $U_q(\mathfrak{g}')$. The
restriction of the module structure on $V(\zeta,\kappa)$ to
$U_q(\mathfrak{g}')$ recovers $V(\zeta)$.
We observe two facts. 
First, because of Theorem~\ref{thm:homtwodimbis},
this restriction induces an isomorphism
\[\operatorname{Hom}_{U_q(\mathfrak{g})}(V(\zeta,\kappa)\otimes V(\zeta',\kappa'),V(\zeta',\kappa')\otimes V(\zeta,\kappa))
\xrightarrow{\sim} \operatorname{Hom}_{U_q(\mathfrak{g}')}(V(\zeta)\otimes V(\zeta'),V(\zeta')\otimes V(\zeta)),\]
since both spaces are two-dimensional and spanned by 
the intertwining operators 
$$ R,R':V(\zeta,\kappa)\otimes V(\zeta',\kappa')\to V(\zeta',\kappa')\otimes V(\zeta,\kappa). $$
Second, the matrix coefficients $a_1,a_2,b_1,b_2,c_1,c_2$ for $R$
and $R'$ given by \cref{thm:homtwodim} only involve $\zeta$ and not $\kappa$. Because of these facts,
there is no difference, for our purposes, between using $V(\zeta,\kappa)$ and
$V(\zeta)$.  In this section we will work with the modules~$V(\zeta)$.

We will now state two results regarding which of the homomorphisms $V (\zeta) \otimes V (\zeta') \longrightarrow V(\zeta') \otimes V (\zeta)$ in the two-dimensional space spanned by $R$ and $R'$ from \cref{thm:homtwodim} satisfy a Yang--Baxter equation.

\begin{theorem}[Yang--Baxter equation I]
	\label{thm:ybeone}
        Let $\zeta$, $\zeta'$ and $\zeta''$ be nonzero complex numbers different from $\pm1$.
	Let $V = V (\zeta)$,  $U = V (\zeta')$ and $W = V (\zeta'')$.
	Then
  \begin{equation}
		\label{eq:srybe} (R_{\zeta',\zeta''} \otimes I_{V(\zeta)})  (I_{V(\zeta')} \otimes R_{\zeta,\zeta''})  
		(R_{\zeta,\zeta'} \otimes I_{V(\zeta'')}) = 
		(I_{V(\zeta'')} \otimes R_{\zeta,\zeta'})  (R_{\zeta,\zeta'' } \otimes I_{V(\zeta')})  (I_{V(\zeta)} \otimes R_{\zeta',\zeta''})
  \end{equation}
  in $\mathrm{Hom} (V \otimes U \otimes W, W \otimes U \otimes V)$. The same
  identity is true with $R'$ replacing $R$.
\end{theorem}

\begin{proof}
  We have from~\cref{thm:homtwodim} that
  \[ \begin{array}{l}
       R_{\zeta,\zeta'} (x \otimes x') = x' \otimes x,\\
       R_{\zeta,\zeta'}  (x \otimes y') = \zeta y' \otimes x,\\
       R_{\zeta,\zeta'}  (y \otimes x') = \zeta' x' \otimes y + (1 - \zeta^2) y'
       \otimes x,\\
       R_{\zeta,\zeta'}  (y \otimes y') = - \zeta \zeta' y' \otimes y,
     \end{array} \]
  with similar identities for $R_{\zeta,\zeta''}$ and $R_{\zeta',\zeta''}$. From this
  (\ref{eq:srybe}) may be checked on each basis vector of $V \otimes U \otimes
  W$. Taking inverses on (\ref{eq:srybe}) and using (\ref{eq:rtwoone}) gives
  the corresponding result for~$R'$.
\end{proof}

Let $\alpha, \beta$ be complex numbers and define the ``baxterized'' $R$-matrix
\begin{equation}
  \label{eq:RRp-combination}
  R (\zeta, \zeta', \alpha, \beta) := \alpha R_{\zeta, \zeta'} + \beta R'_{\zeta, \zeta'}.
\end{equation}
This is the multi-parameter generalization of the $R$-matrix discussed in the introduction and known as
baxterization. We next confirm that it continues, under a mild hypothesis, to satisfy the Yang-Baxter equation.

\begin{theorem}[Yang--Baxter equation II]
  \label{thm:YBE2}
  Assume that $a d e = b c f$. Then
  \begin{equation}
    \label{eq:adebcfybe}
    \begin{multlined}
      \bigl(R (\zeta', \zeta'', a, b) \otimes I_{V (\zeta)}\bigr) \bigl(I_{V (\zeta')}
      \otimes R (\zeta, \zeta'', c, d)\bigr) \bigl(R (\zeta, \zeta', e, f) \otimes I_{V
      (\zeta'')}\bigr) = \\
      \bigl(I_{V (\zeta'')} \otimes R (\zeta, \zeta', e, f)\bigr) \bigl(R (\zeta, \zeta'', c,
      d) \otimes I_{V (\zeta')}\bigr) \bigl(I_{V (\zeta)} \otimes R (\zeta', \zeta'', a,
      b)\bigr).
    \end{multlined}
  \end{equation}
  This is an identity of $U_q (\mathfrak{g})$-homomorphisms
  \[ V (\zeta) \otimes V (\zeta') \otimes V (\zeta'') \longrightarrow V
     (\zeta'') \otimes V (\zeta') \otimes V (\zeta) . \]
\end{theorem}

\begin{proof}
  This can be checked by direct calculation, by applying both sides to eight
  basis vectors of $V (\zeta) \otimes V (\zeta') \otimes V (\zeta'')$.
  
  In the special case $\zeta = \zeta' = \zeta''$ we will give an alternative
  proof based on the skein relation. On this assumption, we will abbreviate $V
  = V (\zeta)$ and $R = R_{\zeta, \zeta}$, $R' = R'_{\zeta, \zeta}$ for the
  rest of the proof.
  
  We will expand both sides of (\ref{eq:adebcfybe}) and extract the
  coefficients of 8 monomials in $a, b, c, d, e, f$. For each of the monomials
  $a c e$, $a c f$, $b c e$, $b d e$, $a d f$ and $b d f$, we deduce the
  equality of the contributions on the left- and right-hand sides of
  (\ref{eq:adebcfybe}) from the Yang-Baxter equation in
  Theorem~\ref{thm:ybeone}.
  
  For example, comparing the coefficients of $a c e$ leads to the requirement
  that
  \begin{equation}
    \label{eq:aceybe} (R \otimes I) (I \otimes R) (R \otimes I) = (I \otimes
    R) (R \otimes I) (I \otimes R) .
  \end{equation}
  This is an immediate consequence of Theorem~\ref{thm:ybeone}. Also,
  comparing the coefficients of $b d f$ is also immediately a consequence of
  Theorem~\ref{thm:ybeone}, but using the Yang-Baxter equation for~$R'$.
  
  Comparing the coefficients of $a c f$ is slightly harder. Then we need
  \begin{equation}
    \label{eq:acfybe} (R \otimes I) (I \otimes R) (R' \otimes I) = (I \otimes
    R') (R \otimes I) (I \otimes R) .
  \end{equation}
  To prove this, we start with
  \[ (I \otimes R) (R \otimes I) (I \otimes R) = (R \otimes I) (I \otimes R)
     (R \otimes I), \]
  which is an instance of Theorem~\ref{thm:ybeone}. Rearranging
  \[ (R \otimes I) (I \otimes R) (R^{- 1} \otimes I) = (I \otimes R^{- 1}) (R
     \otimes I) (I \otimes R) . \]
  Now using (\ref{eq:rtwoone}) this gives (\ref{eq:acfybe}). We leave the
  coefficients of $b c e$, $b d e$, $a d f$ to the reader.
  
  In order to get cancellation for the remaining coefficients $a d e$ and $b c
  f$, the two coefficients must be combined. Since we are assuming that $a d e
  = b c f$, we are reduced to showing that
  \begin{equation}
    \label{eq:hardterms} \begin{alignedat}{3}
      &(R \otimes I) (I \otimes R') (R \otimes I) &-& (I \otimes R) (R' \otimes
      I) (I \otimes R) & \\[-0.25em]
      {}+{} &(R' \otimes I) (I \otimes R) (R' \otimes I) &-& (I \otimes R') (R
      \otimes I) (I \otimes R') = 0.
    \end{alignedat}
  \end{equation}
  We add the following expression to (\ref{eq:hardterms}):
  \begin{equation}
    \label{eq:hardbuf} \begin{array}{c}
      (R \otimes I) (I \otimes R) (R \otimes I) - (I \otimes R) (R \otimes I)
      (I \otimes R)\\
      {}+ (R' \otimes I) (I \otimes R') (R' \otimes I) - (I \otimes R') (R'
      \otimes I) (I \otimes R').
    \end{array}
  \end{equation}
  Note that this vanishes by (\ref{eq:aceybe}) and its analogue for $R'$. Now
  using the skein relation (\ref{eq:skein}), the sum of (\ref{eq:hardterms}) and (\ref{eq:hardbuf}) is $1-\zeta^2$ times
  \[ \begin{array}{c}
       (R \otimes I) (I \otimes I \otimes I) (R \otimes I) - (I \otimes R) (I
       \otimes I \otimes I) (I \otimes R)\\
       {}+ (R' \otimes I) (I \otimes I \otimes I) (R' \otimes I) - (I \otimes
       R') ( I \otimes I \otimes I) (I \otimes R'),
     \end{array} \]
  in other words $1 - \zeta^2$ times
  \[ (R^2 \otimes I) - (I \otimes R^2) + (R')^2 \otimes I - I \otimes (R')^2 . \]
  Now using (\ref{eq:doubleskein}) the first and third terms together produce
  the scalar $(1 - \zeta^2) (1 + \zeta^4)$, which is cancelled by the second
  and fourth terms.
\end{proof}

\section{The Boltzmann weights of the Tokuyama lattice model\label{sec:tokuyama}}
In this section we will show that a particular specialization of the R-matrix \eqref{eq:RRp-combination} used in \cref{thm:YBE2} recovers the Boltzmann weights for the Tokuyama lattice models constructed in~\cite{HamelKingBijective, BrubakerBumpFriedbergHkice}.
The associated partition functions are deformations of the Weyl character formula for $\GL_r$ found by Tokuyama in~\cite{Tokuyama} and each consists of a Schur polynomial $s_\lambda$ times a deformed Weyl denominator  where $\lambda = (\lambda_1,\ldots,\lambda_r) \in \mathbb{Z}^r$ with $\lambda_1 \geq \lambda_2 \geq \cdots \geq \lambda_r \geq 0$.
Using a parameter $v$, this deformation interpolates between the standard Weyl character (or determinant) formula for $s_\lambda$ and the combinatorial formula in terms of semistandard Young tableaux.
The Tokuyama formula was expressed as a weighted sum over Gelfand--Tsetlin patterns and, viewing the sum as a partition function, these Gelfand--Tsetlin patterns are in a weight-preserving bijection with the states of the lattice model.

We will compare our Boltzmann weights with those of~\cite{BrubakerBumpFriedbergHkice} which have two different lattice models, called $\Gamma$ and $\Delta$, computing the same partition function. 
The equality of the $\Gamma$ and $\Delta$ partition functions can be proven using Yang--Baxter equations as in~\cite[Proposition~13]{BrubakerBumpFriedbergHkice}.

The lattice models in~\cite{BrubakerBumpFriedbergHkice} consist of a square grid with $r$ rows and at least $\lambda_1 + r$ columns.
Each vertex has four adjacent edges which means that we get edges around the boundary with only one endpoint attached to a vertex.
We call these edges \emph{boundary edges} in contrast to \emph{interior edges}.
A state of the lattice model corresponds to assigning an attribute to each edge from predefined sets of attributes called \emph{spin sets}.
In the case of~\cite{BrubakerBumpFriedbergHkice} both horizontal and vertical edges have the same binary spin set which we will denote by $\{\oplus, \ominus\}$ meaning plus and minus respectively.
The four edges adjacent to a vertex can be assigned according to one of six possible \emph{vertex configurations} shown in \cref{tab:boltzmann}.
The table also lists the Boltzmann weight of each configuration depending on the deformation parameter $v$ and on a complex parameter $z$ which at row $i$ is replaced by the row parameter $z_i$.
We index the rows from $1$ to $r$ starting from the top for $\Gamma$ and starting from the bottom for $\Delta$.
The Boltzmann weight of a whole state is then the product of the Boltzmann weights of its vertex configurations, and the partition function is the sum of the Boltzmann weights of all the states with some given boundary conditions. 

In this paper we use a description of the $\Gamma$ and $\Delta$ versions of the Tokuyama model such that their admissible states and the set of relevant boundary conditions are the same; they only differ in their Boltzmann weights.
The same description is used in the published version of~\cite{BrubakerBumpFriedbergHkice} while the preprint has different vertex configurations for $\Gamma$ and $\Delta$.

For our boundary conditions, we take the left and bottom boundary edges to be all $\oplus$ and the right boundary to be $\ominus$ on all $r$ rows.
For the top boundary we put $\ominus$ at the $r$ column positions labeled by integers $\lambda + \rho$ where $\rho = (r-1, r-2, \ldots, 1, 0)$ assures distinct columns. The columns are indexed from right to left starting from $0$.
For the remaining top boundary edges we take $\oplus$.
To summarize, the boundary conditions are fully specified by partitions $\lambda$ which determine the top boundary positions.

We denote the sets of admissible states with these boundary conditions by~$\mathfrak{S}_\lambda$.
For the $\Gamma$ and $\Delta$ versions we denote the corresponding partition functions by $Z^\Gamma_\lambda(\mathbf{z})$ and $Z^\Delta_\lambda(\mathbf{z})$ respectively where $\mathbf{z} = (z_1, \ldots, z_r) \in (\mathbb{C}^\times)^r$.
The boundary conditions and the vertex configurations in \cref{tab:boltzmann} are such that the $\ominus$ edges in an admissible state form paths through the lattice, and they always travel down or to the right at each vertex going from the top boundary to the right boundary.

Note that the weights in \cref{tab:boltzmann} differ slightly from those in~\cite{BrubakerBumpFriedbergHkice}: we have set all the $t_i$ in~\cite[Table~2]{BrubakerBumpFriedbergHkice} to be equal to $-v$ and a factor $\sqrt{v}$ (highlighted in blue in the table) has been shifted between the $\texttt{b}_1$ and $\texttt{b}_2$ weights.

\begin{table}[htpb]
  \centering
  \caption{Vertex configurations and Boltzmann weights for the $\Gamma$ and $\Delta$ variants of the Tokuyama model that will be used in this paper.
    They are based on \cite{BrubakerBumpFriedbergHkice} but differ slightly as described in the surrounding text and in particular we have introduced extra factors of $\sqrt{v}$ which are shown in blue.
  For readability we have marked the $\ominus$ spins with bold red edges since these edges will form paths in the lattice model.} 
  \label{tab:boltzmann}
  \begin{equation*}
    \def\vertexscale{0.75}
    \setlength\extrarowheight{3pt}
    \begin{array}{|c|c|c|c|c|c|c|}
      \hline
      & \texttt{a}_1 & \texttt{a}_2 & \texttt{b}_1 & \texttt{b}_2 & \texttt{c}_1 & \texttt{c}_2 \\ \hline 
      & 
      \begin{tikzpicture}[scale=\vertexscale]
        \draw (-1,0) node[state] {$+$} -- (0,0) node[dot] {} -- (1,0) node[state] {$+$};
        \draw (0,-1) node[state] {$+$} -- (0,1) node[state] {$+$};
      \end{tikzpicture}
      &
      \begin{tikzpicture}[scale=\vertexscale]
        \draw[ultra thick, red] (0,-1) node[state] {$-$} -- (0,1) node[state] {$-$};
        \draw[ultra thick, red] (-1,0) node[state] {$-$} -- (0,0) node[dot] {} -- (1,0) node[state] {$-$};
      \end{tikzpicture}
      &
      \begin{tikzpicture}[scale=\vertexscale]
        \draw[ultra thick, red] (0,-1) node[state] {$-$} -- (0,1) node[state] {$-$};
        \draw (-1,0) node[state] {$+$} -- (0,0) node[dot] {} -- (1,0) node[state] {$+$};
      \end{tikzpicture} 
      &
      \begin{tikzpicture}[scale=\vertexscale]
        \draw[ultra thick, red] (-1,0) node[state] {$-$} -- (0,0) node[dot] {} -- (1,0) node[state] {$-$};
        \draw (0,-1) node[state] {$+$} -- (0,1) node[state] {$+$};
      \end{tikzpicture} 
      &
      \begin{tikzpicture}[scale=\vertexscale]
        \draw[ultra thick, red] (-1,0) node[state] {$-$} -- (0,0) node[dot] {} -- (0,-1) node[state] {$-$};
        \draw (0,1) node[state] {$+$} -- (0,0) -- (1,0) node[state] {$+$};
      \end{tikzpicture} 
      &
      \begin{tikzpicture}[scale=\vertexscale]
        \draw (-1,0) node[state] {$+$} -- (0,0) node[dot] {} -- (0,-1) node[state] {$+$};
        \draw[ultra thick, red] (0,1) node[state] {$-$} -- (0,0) -- (1,0) node[state] {$-$};
      \end{tikzpicture} 
      \rule{0pt}{5em}\\ \hline
      \Gamma & 1 & z & -v/{\color{blue}\sqrt{v}} & {\color{blue}\sqrt{v}} z & (1-v)z & 1 \\ \hline
      \Delta & 1 & -vz & 1/{\color{blue}\sqrt{v}} & {\color{blue}\sqrt{v}} z & (1-v)z & 1 \\ \hline
    \end{array}
  \end{equation*}
  
\end{table}

The shifted factor $\sqrt{v}$ changes the partition function compared to \cite{BrubakerBumpFriedbergHkice} with an overall factor.
That is, we still get the Tokuyama formula with a Schur function $s_\lambda$ and a deformed Weyl denominator factor, but now with a new pre-factor $v^{\abs{\lambda}/2}$ where $\abs{\lambda} = \sum_{i=1}^r \lambda_i$. 
\begin{proposition}
  The partition functions with Boltzmann weights as in \cref{tab:boltzmann} are
  \begin{equation}
    Z^\Gamma_\lambda(\mathbf{z}) = Z^\Delta_\lambda(\mathbf{z}) = v^{\abs{\lambda}/2} \prod_{i<j} (z_i - v z_j) s_\lambda(\mathbf{z}).
  \end{equation}
\end{proposition}
\begin{proof}
  The $\Gamma$ and $\Delta$ partition functions for the Boltzmann weights in \cite{BrubakerBumpFriedbergHkice} were there both computed as $\prod_{i<j} (z_i - v z_j) s_\lambda(\mathbf{z})$.
  It is thus enough to show that, for each state $\mathfrak{s}$ in $\mathfrak{S}_\lambda$, if we multiply the $\texttt{b}_1$ and $\texttt{b}_2$ vertex configurations by $1/\sqrt{v}$ and $\sqrt{v}$ respectively we get the same total factor of $v^{\abs{\lambda}/2}$.

  We will make an argument by considering the $r$ paths through the lattice separately.
  To then avoid ambiguities at $\texttt{a}_2$ configurations we will in this proof consider the paths to only be meeting and not crossing in that case.
  
  We will deform the $r$ paths of each state to canonical configurations in such a way that the extra $v$-factors remain unchanged and where the extra $v$-factors for the resulting canonical configurations are easy to determine. 
  From the weights in \cref{tab:boltzmann} it is not difficult to check that the total extra $v$-factors from the $\texttt{b}_1$ and $\texttt{b}_2$ configurations remain unchanged if we make the following change to a corner configuration
  \begin{equation}
    \begin{tikzpicture}[baseline=1.5em, scale=0.75]
      \draw[dashed] (-1,0) -- (3,0);
      \draw[dashed] (-1,2) -- (3,2);
      \draw[dashed] (0,-1) -- (0,3);
      \draw[dashed] (2,-1) -- (2,3);
      \draw[ultra thick, red] (0,2) -- (2,2) node[midway,state] {$-$} -- (2,0) node[midway,state] {$-$};
      \draw (0,2) -- (0,0) node[midway,state] {$+$} -- (2,0) node[midway,state] {$+$};
      \node[dot] at (0,0){};
      \node[dot] at (0,2){};
      \node[dot] at (2,0){};
      \node[dot] at (2,2){};
    \end{tikzpicture} 
    \quad
    \longrightarrow
    \quad
    \begin{tikzpicture}[baseline=1.5em, scale=0.75]
      \draw[dashed] (-1,0) -- (3,0);
      \draw[dashed] (-1,2) -- (3,2);
      \draw[dashed] (0,-1) -- (0,3);
      \draw[dashed] (2,-1) -- (2,3);
      \draw[ultra thick, red] (0,2) -- (0,0) node[midway,state] {$-$} -- (2,0) node[midway,state] {$-$};
      \draw (0,2) -- (2,2) node[midway,state] {$+$} -- (2,0) node[midway,state] {$+$};
      \node[dot] at (0,0){};
      \node[dot] at (0,2){};
      \node[dot] at (2,0){};
      \node[dot] at (2,2){};
    \end{tikzpicture}  
  \end{equation}
  where the dashed boundary edges are any admissible spins that are kept unchanged.
  Indeed, under this transformation the upper left vertex will always gain a $v$-factor of $1/\sqrt{v}$ and the lower right vertex will always gain a $v$-factor of $\sqrt{v}$ while the other vertices do not change any $v$-factors.

  By repeatedly using this transformation we may deform the left-most path to an L-shape with unchanged top and right boundary positions.
  The same procedure can be done for all paths from left to right and we end up in a canonical configuration of L-shapes with corners at columns $\lambda_1 + r - 1, \lambda_2 + r - 2, \ldots, \lambda_r$ and rows $r, r-1, \ldots, 1$.
  Since columns are counted from $0$ and rows from $1$, the number of $\texttt{b}_1$ and $\texttt{b}_2$ patterns for path $i\in\{1,\ldots,r\}$ are thus $r-i$ and $\lambda_i - (r - i)$ respectively giving a total extra factor $v^{\abs{\lambda}/2}$ for the whole state.
\end{proof}

As mentioned earlier, the Tokuyama lattice models were shown in~\cite{BrubakerBumpFriedbergHkice} to satisfy several Yang--Baxter equations.
The Yang--Baxter equation for three linear maps $R : U \otimes V \to V \otimes U$, $S : U \otimes W \to W \otimes U$ and $T : V \otimes W \to W \otimes V$ can be expressed as
\begin{equation}
  \label{eq:RST}
  (\id_W \otimes R)(S \otimes \id_V)(\id_U \otimes T) = (T \otimes \id_U)(\id_V \otimes S)(R \otimes \id_W)
\end{equation}
as maps $U \otimes V \otimes W \to W \otimes V \otimes U$.
We will sometimes denote the difference between the left-hand side and the right-hand side of this Yang--Baxter equation by $\llbracket R, S, T \rrbracket$ such that~\eqref{eq:RST} may be written as $\llbracket R, S, T \rrbracket = 0$. 
The Yang--Baxter equation is often also expressed in terms of endomorphisms in $\operatorname{End}(U \otimes V)$, $\operatorname{End}(U \otimes W)$ and $\operatorname{End}(V \otimes W)$ and then \eqref{eq:RST} takes the form of \eqref{eq:rstversion} in Section~\ref{sec:affine}.

The Yang--Baxter equation has a pictorial description in terms of lattice models where each map $R$, $S$ and $T$ corresponds to a type of vertex and where the modules $U$, $V$ and $W$ correspond to (spans of) spin sets for the corresponding edges adjoining these vertices.
\begin{equation}
  \label{eq:YBE-picture}
  \begin{tikzpicture}[baseline=-1mm, scale=0.75]
    \draw (-1, 1) node[label=left:$V$]{} -- (1,-1) -- (3,-1) node[label=right:$V$]{};
    \draw (-1,-1) node[label=left:$U$]{} -- (1, 1) -- (3, 1) node[label=right:$U$]{};
    \draw (2,2) node[label=above:$W$]{} -- (2,-2) node[label=below:$W$]{};
    \node[label=above:$R$, dot] at (0,0) {};
    \node[label=north east:$S$, dot] at (2,1) {};
    \node[label=north east:$T$, dot] at (2,-1) {};
  \end{tikzpicture}
  \quad
  =
  \quad
  \begin{tikzpicture}[baseline=-1mm,xscale=-0.75,yscale=-0.75]
    \draw (-1, 1) node[label=right:$V$]{} -- (1,-1) -- (3,-1) node[label=left:$V$]{};
    \draw (-1,-1) node[label=right:$U$]{} -- (1, 1) -- (3, 1) node[label=left:$U$]{};
    \draw (2,2) node[label=below:$W$]{} -- (2,-2) node[label=above:$W$]{};
    \node[label=above:$R$, dot] at (0,0) {};
    \node[label=north east:$S$, dot] at (2,1) {};
    \node[label=north east:$T$, dot] at (2,-1) {};
  \end{tikzpicture}
\end{equation}

It was shown in~\cite{BBB} (using more general lattice models including the Tokuyama model as a special case) that the quantum group associated to the Tokuyama model should be the affine supersymmetric quantum group $U_q(\widehat{\mathfrak{gl}}(1|1))$ where $q=\sqrt{v}$, and that the spin set for the horizontal edges at row $i$ furnishes an evaluation module $V_{z_i}$.
In more detail, they showed the following.
Interpret the $R$ in \eqref{eq:YBE-picture} as a new kind of vertex, called an $R$-vertex, connecting to four horizontal edges and let its weights be the $R$-matrix elements computed from two evaluation modules $U = V_{z_1}$ and $V = V_{z_2}$ of $U_q(\widehat{\mathfrak{gl}}(1|1))$.
Then \eqref{eq:YBE-picture} is satisfied if you interpret $S$ and $T$ as the standard vertices of \cref{tab:boltzmann} for rows $i$ and $j$.
We call the standard vertices of the lattice with two horizontal edges and two vertical edges $T$-vertices.
The pictorial description~\eqref{eq:YBE-picture} is then stating an equality of partition functions obtained from the vertices on each side of the equation given fixed boundary edges.

While $U$ and $V$ have been shown to be evaluations modules it has long been an open question what the quantum group module $W$ is for the vertical edges, and if the Boltzmann weights for the $T$-vertices in~\cite{BrubakerBumpFriedbergHkice} can be recovered from quantum group considerations alone, either from a universal $R$-matrix or other methods.

In~\cite{BrubakerBumpFriedbergHkice} the weights for the $T$-vertices were obtained by other means.
For example, the Boltzmann weights can be obtained by comparing with the weights for the Gelfand--Tsetlin patterns in the Tokuyama formula (with some potential adjustments for solvability).

As mentioned above, there are two types of $T$-weights $T_X$ with $X \in \{\Gamma, \Delta\}$, but if we allow the horizontal modules $U$ and $V$ in \eqref{eq:YBE-picture} to be different we get four different types of $R$-weights $R_{XY}$ with $X,Y \in \{\Gamma, \Delta\}$.
The different vertex configurations and weights are shown in~\cref{tab:boltzmann} for $T_X$ and in~\cref{tab:R-vertices} for $R_{XY}$.
Here the $R_{\Gamma\Gamma}$ and $R_{\Delta\Delta}$ are the standard $R$-matrices for $\Gamma$ and $\Delta$ Tokuyama ice as introduced before, but there are also the mixed types $R_{\Gamma\Delta}$ and $R_{\Delta\Gamma}$.
The latter were used to show the equality of the $\Gamma$ and $\Delta$ partition functions in~\cite[Proposition~13]{BrubakerBumpFriedbergHkice} as mentioned above.

\begin{table}[htpb]
  \centering
  \caption{Vertex configurations and weights for the different types of $R$-vertices for the Tokuyama model.
  For readability we have marked the $\ominus$ spins with bold red edges since these edges will form paths in the lattice model.
  The row parameters $z_1$ and $z_2$ are associated to the evaluation modules $U$ and $V$ respectively as illustrated in the first column.
The factors of $\sqrt{v}$ marked in blue are needed for the similarly introduced factors for the $T$-vertices. 
In particular, for $\Gamma\Gamma$ and $\Delta\Delta$ the $\texttt{b}_1$ and $\texttt{b}_2$ weights are the same.} 
  \label{tab:R-vertices}
  \begin{equation*}
    \def\vertexscale{0.5}
    \setlength\extrarowheight{3pt}
    \begin{array}{|c|c|c|c|c|c|c|}
      \hline
      & \texttt{a}_1 & \texttt{a}_2 & \texttt{b}_1 & \texttt{b}_2 & \texttt{c}_1 & \texttt{c}_2 \\ \hline 
      \begin{tikzpicture}[scale=\vertexscale]
        \node at (0,-1) {$z_1$};
        \node at (0,1) {$z_2$};
      \end{tikzpicture}
      & 
      \begin{tikzpicture}[scale=\vertexscale]
        \draw (-1,-1) node[state] {$+$} -- (0,0) node[dot] {} -- (1,1) node[state] {$+$};
        \draw (-1,1) node[state] {$+$} -- (1,-1) node[state] {$+$};
      \end{tikzpicture}
      &
      \begin{tikzpicture}[scale=\vertexscale]
        \draw[ultra thick, red] (-1,-1) node[state] {$-$} -- (1,1) node[state] {$-$};
        \draw[ultra thick, red] (-1,1) node[state] {$-$} -- (0,0) node[dot] {} -- (1,-1) node[state] {$-$};
      \end{tikzpicture}
      &
      \begin{tikzpicture}[scale=\vertexscale]
        \draw[ultra thick, red] (-1,1) node[state] {$-$} -- (1,-1) node[state] {$-$};
        \draw (-1,-1) node[state] {$+$} -- (0,0) node[dot] {} -- (1,1) node[state] {$+$};
      \end{tikzpicture} 
      &
      \begin{tikzpicture}[scale=\vertexscale]
        \draw[ultra thick, red] (-1,-1) node[state] {$-$} -- (0,0) node[dot] {} -- (1,1) node[state] {$-$};
        \draw (-1,1) node[state] {$+$} -- (1,-1) node[state] {$+$};
      \end{tikzpicture} 
      &
      \begin{tikzpicture}[scale=\vertexscale]
        \draw[ultra thick, red] (-1,-1) node[state] {$-$} -- (0,0) node[dot] {} -- (1,-1) node[state] {$-$};
        \draw (-1,1) node[state] {$+$} -- (0,0) -- (1,1) node[state] {$+$};
      \end{tikzpicture} 
      &
      \begin{tikzpicture}[scale=\vertexscale]
        \draw (-1,-1) node[state] {$+$} -- (0,0) node[dot] {} -- (1,-1) node[state] {$+$};
        \draw[ultra thick, red] (-1,1) node[state] {$-$} -- (0,0) -- (1,1) node[state] {$-$};
      \end{tikzpicture} 
      \rule{0pt}{5em}\\ \hline
      \Gamma\Gamma & z_2-v z_1 & z_1-v z_2 & (v/{\color{blue}\sqrt{v}})\,(z_1-z_2) & {\color{blue}\sqrt{v}}\,(z_1-z_2) & (1-v)z_1 & (1-v)z_2 \\ \hline
      \Delta\Delta  & z_1-v z_2 & z_2-v z_1 & (v/{\color{blue}\sqrt{v}})\,(z_1-z_2) & {\color{blue}\sqrt{v}}\,(z_1-z_2) & (1-v)z_1 & (1-v)z_2 \\ \hline
      \Gamma\Delta  & z_1-v z_2 & z_1-v z_2 & (1/{\color{blue}\sqrt{v}})\,(v^2z_2-z_1) & {\color{blue}\sqrt{v}}\,(z_1-z_2) & (1-v)z_1 & (1-v)z_2 \\ \hline
      \Delta\Gamma & z_2-v z_1 & z_2-v z_1 & (1/{\color{blue}\sqrt{v}})\,(z_2-v^2z_1) & {\color{blue}\sqrt{v}}\,(z_1-z_2) & (1-v)z_1 & (1-v)z_2 \\ \hline      
    \end{array}
  \end{equation*}
  
\end{table}

\begin{proposition}
  For all $X,Y,Z \in \{\Gamma, \Delta\}$ the following Yang--Baxter equations hold
  \begin{equation}
    \llbracket R_{XY}(z_1, z_2), T_X(z_1), T_Y(z_2) \rrbracket = 0 \qquad \llbracket R_{XY}(z_1,z_2), R_{XZ}(z_1,z_3), R_{YZ}(z_2,z_3) \rrbracket = 0.
  \end{equation}
\end{proposition}
\begin{proof}
  The proof amounts to checking a finite (small) set of symbolic equations which can easily be done using Computer Algebra Software such as \texttt{SageMath}.
  Various subsets of these proofs also appear as (special cases) in for example~\cite{BrubakerBumpFriedbergHkice, BBB}.
\end{proof}

The following theorem shows that if we take $R(\zeta, \zeta', \alpha, \beta)$ defined in \eqref{eq:RRp-combination} as a linear combination of the two solutions $R$ and $R'$ in \cref{thm:homtwodim} then we can recover all the different $R$- and $T$-weights for both the $\Gamma$- and $\Delta$-versions of the Tokuyama lattice model.

\begin{theorem} \label{thm:alltheice}
  Let $R(\zeta, \zeta', \alpha, \beta)$ be the linear combination $\alpha R_{\zeta,\zeta'} + \beta R'_{\zeta,\zeta'}$ defined in \eqref{eq:RRp-combination}.
   Then the following specializations obtain all the different R- and T-weights of the Tokuyama lattice model of \cref{tab:boltzmann,tab:R-vertices}. 
   \begin{equation} \label{ybsystemmatches}
     \begin{array}{|c|c|c|c|c|}
       \hline
       R (\zeta, \zeta', \alpha, \beta) & \zeta & \zeta' & \alpha & \beta\\
       \hline\hline
       T_{\Gamma} & v^{1 / 2} & 0 & z & 1\\
       \hline
       T_{\Delta} & - v^{- 1 / 2} & 0 & - v z & 1\\
       \hline
       R_{\Gamma \Gamma} & v^{1 / 2} & v^{1 / 2} & z_1 & z_2\\
       \hline
       R_{\Delta \Delta} & - v^{- 1 / 2} & - v^{- 1 / 2} & - v z_1 & - v z_2\\
       \hline
       R_{\Gamma \Delta} & v^{1 / 2} & - v^{- 1 / 2} & z_1 & - v z_2\\
       \hline
       R_{\Delta \Gamma} & - v^{- 1 / 2} & v^{1 / 2} & - v z_1 & z_2\\
       \hline
     \end{array}
   \end{equation}
\end{theorem}
Note the patterns for $(\zeta, \alpha)$ and $(\zeta',\beta)$ asociated to $\Gamma$ and $\Delta$ rows.
We are specializing $\zeta'\to 0$. Now $V(\zeta')$ is not defined if $\zeta'=0$, though it can be
extended to $0$ if we instead regard $V(\zeta')$ as a module of the Yangian. This is not necessary
since the Yang-Baxter equation can be deduced by continuity at $\zeta'=0$.

\begin{proof}
  The statement follows from comparing $R(\zeta, \zeta',\alpha,\beta)$ with \cref{tab:boltzmann,tab:R-vertices} using the following conventions.
  The basis vectors $x$ and $y$ of \cref{prop:basis} correspond to the edge spins $\ominus$ and $\oplus$ respectively and the weights are read out using tensored inputs $a \otimes b$ and outputs $c \otimes d$ according to 
  \begin{equation}
    \def\vertexscale{0.75}
    \begin{tikzpicture}[scale=\vertexscale, baseline]
      \draw (-1,0) node[state] {$d$} -- (0,0) node[dot] {} -- (1,0) node[state] {$a$};
      \draw (0,-1) node[state] {$b$} -- (0,1) node[state] {$c$};
      \draw[thick, decorate, decoration={brace, mirror, raise=12pt}] (0,-1) -- (1,0) node[pos=0.5, below=14pt, sloped] {input}; 
      \draw[thick, decorate, decoration={brace, mirror, raise=12pt}] (0,1) -- (-1,0) node[pos=0.5, above=14pt, sloped] {output};
    \end{tikzpicture}
    \qquad \qquad
    \begin{tikzpicture}[scale=\vertexscale, baseline]
      \draw (-1,-1) node[state] {$d$} -- (0,0) node[dot] {} -- (1,1) node[state] {$a$};
      \draw (1,-1) node[state] {$b$} -- (-1,1) node[state] {$c$};
      \draw[thick, decorate, decoration={brace, raise=12pt}] (1,1) -- (1,-1) node[pos=0.5, above=14pt, sloped] {input}; 
      \draw[thick, decorate, decoration={brace, mirror, raise=12pt}] (-1,1) -- (-1,-1) node[pos=0.5, below=14pt, sloped] {output};
    \end{tikzpicture}
  \end{equation}
  The module isomorphisms labelings $a_1$, $a_2$, $b_1$, $b_2$, $c_1$ and $c_2$ in \eqref{eq:phiabcd} for $R$ and $R'$ then agrees with the labelings of the vertex configurations in~\cref{tab:boltzmann,tab:R-vertices}.
  Using the table in \cref{thm:homtwodim} the statement is then easily verified. 
\end{proof}

\section{Coincidence with Affine $R$-matrices\label{sec:affine}}

Having built families of $R$-matrices from pairs of modules for $U_q(\mathfrak{g})$, the next two sections explore the extent to which these families exhaust or lead to interesting solutions to parametrized Yang-Baxter equations from other sources. In this section, we explore this question for modules of the affine superalgebra $U_q(\hat{\mathfrak{g}})$. Indeed, any such triple of modules (call the underlying vector spaces $U, V, W,$ respectively) gives rise to $R$-matrices for each of the three pairs and corresponding Yang-Baxter equations. If we name these endomorphisms $R \in \textrm{End}(U \otimes V)$, $S \in \textrm{End}(U \otimes W)$, and $T \in \textrm{End}(V \otimes W)$, then in the larger space $\textrm{End}(U \otimes V \otimes W)$, we have
\begin{equation} R_{12} S_{13} T_{23} = T_{23} S_{13} R_{12}, \label{eq:rstversion} \end{equation}
where $R_{12} := R \otimes \textrm{Id}$ denotes the action of $R$ on $U \otimes V$ while acting by the identity on $W$, and similarly for $S_{13}$ and $T_{23}$. So it is natural to ask: to what extent do the solutions of the previous section (which were built solely from information from the finite-dimensional {\it non-affine} modules of $U_q(\mathfrak{g})$) exhaust the set of Yang-Baxter equations arising from the {\it affine} setting of $U_q(\hat{\mathfrak{g}})$ modules? Our first step is to collect several such families from affine superalgebras.

\subsection{$R$-matrices from FRT presentations of affine superalgebras}

Combining the work of Faddeev, Reshetikhin and Takhtajan \cite{FRT1, FRT2} and
Reshetikhin and Semenov-Tian-Shansky \cite{ReshetikhinSTS} gives a realization of the Hopf
algebra $U_q(\hat{\mathfrak{g}})$, the affine Kac-Moody algebra associated to a
simple Lie algebra $\mathfrak{g}$, using solutions of (parametrized)
Yang-Baxter equation. Thus we refer to such a presentation with relations from
Yang-Baxter equations as an ``FRT presentation.'' 
On the other hand, Drinfeld~\cite{DrinfeldYangians} gave another presentation
of these affine quantum groups that includes Yangians and is well adapted to the
study of level 0 representations, such as Kirillov-Reshetikhin modules.
See also~\cite{ChariHernandez}.
Ding and Frenkel \cite{DingFrenkel} showed that the FRT presentation coincides with the
Drinfeld presentation for $U_q(\hat{\mathfrak{g}})$, a presentation in the
style of the Chevalley-Serre presentation for simple Lie algebras.

We review the presentation in the case of the Lie superalgebra $U_q :=
U_q(\widehat{\mathfrak{gl}}(1|1))$ (identified with the Drinfeld-type
presentation in \cite{CaiWangWuZhaoDrinfeld}) as given in \cite{HuafengZhangGL11}, and then
proceed to quickly derive examples of $R$-matrices from its modules.
Yamane \cite{LinYamaneZhang} demonstrates the equivalence of the Drinfeld and
Drinfeld-Jimbo presentations. The latter shows that our superalgebra
$U_q(\widehat{\mathfrak{gl}}(1|1))$ as defined in Section~\ref{sec:gl11} is a 
Hopf subalgebra.

Let $R(z,w)$ be the Perk-Schultz matrix, viewed as an element of $\textrm{End} (V \otimes V)$ where $V = \mathbb{C} v_1 \oplus \mathbb{C} v_2$ is a two-dimensional super vector space with $\mathbb{Z} / 2 \mathbb{Z}$ grading $|v_1| = 0$, the even graded piece, and $|v_2| = 1$, the odd graded piece. It takes the form
\begin{multline} R(z,w) = \sum_{i=1}^2 (z q_i - w q_i^{-1}) E_{ii} \otimes E_{ii} + (z-w) \sum_{i \ne j} E_{ii} \otimes E_{jj} + \\
z (q - q^{-1}) E_{21} \otimes E_{12} + w (q^{-1} - q) E_{12} \otimes E_{21}
\end{multline}
where $q_1 = q$, $q_2 = q^{-1}$ and $E_{ij}$ is the endomorphism with $E_{ij}(v_k) = \delta_{jk} v_i$, matching left multiplication by the corresponding elementary matrix.
In matrix form, ordering the basis vectors lexicographically $\{ v_1 \otimes v_1, v_1 \otimes v_2, v_2 \otimes v_1, v_2 \otimes v_2 \}$ left-to-right in columns and top-to-bottom in rows:
\begin{equation} R(z,w) = \begin{pmatrix} zq - wq^{-1} & & & \\ & z-w & w(q^{-1} - q) & \\ & z (q - q^{-1}) & z-w & \\ & & & z q^{-1} - w q  \end{pmatrix}. \label{eq:rzw} \end{equation}
Note that $R(z,w)$ is homogeneous of degree one in $z$ and $w$ and differs by a scalar from the $R$-matrix $R(x)$ for the standard representation of $U_q$ with $x = w/z$.
\begin{definition}[Zhang, \cite{HuafengZhangGL11}] 
  \label{def:FRT}
The quantum affine superalgebra $U_q$ is the superalgebra with generators $s^{(n)}_{ij}$ and $t^{(n)}_{ij}$ with $i, j = 1,2$ and $n \geqslant 0$, satisfying the relations:
\begin{eqnarray*} R_{23}(z,w) T_{12}(z) T_{13}(w) & = & T_{13}(w) T_{12}(z) R_{23}(z,w) \\
R_{23}(z,w) S_{12}(z) S_{13}(w) & = & S_{13}(w) S_{12}(z) R_{23}(z,w) \\
R_{23}(z,w) T_{12}(z) S_{13}(w) & = & S_{13}(w) T_{12}(z) R_{23}(z,w)  \end{eqnarray*}
$$ t_{12}^{(0)} = s_{21}^{(0)} = 0, \quad t_{ii}^{(0)} s_{ii}^{(0)} = 1 = s_{ii}^{(0)} t_{ii}^{(0)} \quad \text{for $i=1,2$.} $$
Here $T(z) = \sum_{i,j} t_{ij}(z) \otimes E_{ij} \in (U_q \otimes \operatorname{End}(V))[[z^{-1}]]$ and $t_{ij}(z) = \sum_{n \geqslant 0} t_{ij}^{(n)} z^{-n}$, and similarly for $S(z)$ with $z^{-n}$ replaced by $z^n$. Thus our relations take place in $(U_q \otimes \operatorname{End}(V \otimes V))[[z,z^{-1},w,w^{-1}]]$.
\end{definition}

Note that the version of the Yang-Baxter equations in the definition above differ from those in~\eqref{eq:rstversion}, as Zhang is applying the $R$-matrix to the second and third tensor factors in the triple tensor product. There is a natural co-product for which $U_q$ is a Hopf superalgebra, but we do not explicitly use that here. The fact that this FRT presentation coincides with the Drinfeld presentation for $U_q$ is proved in~\cite{CaiWangWuZhaoDrinfeld}.

\begin{theorem} Let $\pi : U_q \longrightarrow \operatorname{End}(W)$ with $s_{ij}(z) \mapsto \pi(s_{ij}(z))$ and similarly for $t_{ij}(z)$. Then if we define
$$ \pi(S)(z) = \sum_{i,j = 1,2} \pi(s_{ij}(z)) \otimes E_{ij} \in \operatorname{End}(W \otimes V)[[z]], $$
and similarly define $\pi(T)$, then the following relations hold in $\operatorname{End}(W \otimes V \otimes V)[[z,z^{-1},w,w^{-1}]]$:
\begin{eqnarray*} R_{23}(z,w) \pi(T)_{12}(z) \pi(T)_{13}(w) & = & \pi(T)_{13}(w) \pi(T)_{12}(z) R_{23}(z,w) \\
R_{23}(z,w) \pi(S)_{12}(z) \pi(S)_{13}(w) & = & \pi(S)_{13}(w) \pi(S)_{12}(z) R_{23}(z,w) \\
R_{23}(z,w) \pi(T)_{12}(z) \pi(S)_{13}(w) & = & \pi(S)_{13}(w) \pi(T)_{12}(z) R_{23}(z,w)  \end{eqnarray*}
\end{theorem}

\begin{proof} $\pi$ is an algebra map, so since the relations in Definition~\ref{def:FRT} are satisfied in $(U_q \otimes \operatorname{End}(V \otimes V))[[z,z^{-1},w,w^{-1}]]$, they continue to hold upon applying $\pi$ in the resulting space $\operatorname{End}(W \otimes V \otimes V)[[z,z^{-1},w,w^{-1}]]$.
\end{proof}

For example, we may take the representation $\pi_{c,d}$ of $U_q$ on the two-dimensional vector space $W$ with
$$ \pi_{c,d}(s_{ij}(z))_{1 \leqslant i,j \leqslant 2} = \begin{pmatrix} c \frac{1 - zd}{1 - z d c^2} E_{11} + c \frac{q^{-1} - zdq}{1 - zd c^2} E_{22} & c \frac{(q^{-1} - q)(dc^2-d)}{1-zdc^2} E_{12} \\ \frac{-z}{1-zdc^2} E_{21} & E_{11} + \frac{q^{-1}-zdc^2 q}{1-zdc^2} E_{22} \end{pmatrix} $$
and $\pi_{c,d}(t_{ij}(z))$ differs from this by a rational function in $c, d,$ and $z$, independent of~$q$. 
This representation appears on p. 1601 of Zhang \cite{HuafengZhangGL11}. This same representation appears under 
the name $\rho_{a,b}$ on p. 1144 of \cite{HuafengZhangRTT}, with $a = d$ and $b = d c^2$, under a scalar 
change of basis $(v_1, v_2) \mapsto (c v_1, v_2)$. In \cite{HuafengZhangRTT}, it is explained that the 
family $\rho_{a,b}$ with $a, b \in \mathbb{C}$ and $a \ne b$ give all prime, simple modules in the category 
of finite-dimensional representations of the super Yangian $Y_q(\mathfrak{gl}(1|1))$, the subalgebra of $U_q$ 
generated by the $s_{ij}^{(n)}$ for $n \geqslant 0$ and $(s_{ii}^{(0)})^{-1}$. Indeed they give all possible non-trivial 
linear factors in any associated Drinfeld polynomial. This demonstrates that the representations $\pi_{c,d}$ result 
in a rich collection of $U_q$ modules that includes, in particular, Kirillov-Reshetikhin modules in the special case 
where the associated highest weight is a multiple of a fundamental weight.

Let us clear denominators by multiplying by the scalar $1-z dc^2$ and then consider the resulting form of $\pi_{c,d}(S)$:
$$ \begin{pmatrix} c (1 - zd) E_{11} \otimes E_{11} + c (q^{-1} - zdq) E_{22} \otimes E_{11} & c (q^{-1} - q)(dc^2-d) E_{12} \otimes E_{12} \\ -z E_{21} \otimes E_{21} & (1-zdc^2) E_{11} \otimes E_{22} + (q^{-1}-zdc^2 q) E_{22} \otimes E_{22} \end{pmatrix} $$
with corresponding $R$-matrix as an action on $W \otimes V$ 
\begin{equation} \begin{pmatrix} c (1-zd) & & &  \\ & c(q^{-1} - zdq) & c (q^{-1} - q)(dc^2-d) & \\ & -z & 1 - zdc^2 & \\ & & & q^{-1} - z d c^2 q \end{pmatrix}. \label{eq:rcd} \end{equation}
Indeed, this is because the second tensor factor is specifying which connected $2 \times 2$ minor in the above matrix 
that the $\pi(s_{i,j})$ is acting in. By our previous theorem, the above matrix satisfies a Yang-Baxter equation of 
the second type in our list of three relations, using the Perk-Schultz (i.e., $U_q$ fundamental representation) $R$-matrix.
This matrix satisfies a graded solution to the Yang-Baxter equation, and we may obtain an ungraded solution by negating the $(1,1)$ entry above. We refer to the resulting matrix as the ``ungraded'' $R$-matrix.

\begin{proposition} Given any representation $(\pi_{c,d}, W)$ as above, and standard module $V$, the corresponding ungraded $R$-matrix in $\textrm{End}(W \otimes V)$ is a member of the family of $R$-matrices in~\eqref{eq:RRp-combination} for some choice of parametrizing data $\zeta, \zeta', \alpha, \beta$.
\end{proposition}

\begin{proof} 
We may solve for the explicit choice of these parameters by comparing the resulting matrices and we find that setting
$$ \beta = zdc, \quad \alpha = q^{-1}, \quad \zeta = q, \quad \zeta' = c, $$
we obtain equality of the two matrices in the statement of the proposition, up to a change of basis. Indeed the product of the off-diagonal terms agree.
\end{proof}

It may be advantageous, for certain points of view, to have a precise understanding of the algebraic origins of a given solution to the Yang-Baxter equation. But if one simply wants a supply of $R$-matrices satisfying the Yang-Baxter equation, then the previous proposition shows that our method of generating them from linear combinations of modules of the finite algebra produces a large class of $R$-matrices from the affine algebra $U_q$.

\subsection{Tokuyama ice as limiting cases} We now discuss certain limits of the
$R$-matrix appearing in~(\ref{eq:rcd}), which arose in the
previous section in connection with certain lattice models used to represent
generating function identities as in \cite{Tokuyama}. First consider the
specialization of the $R$-matrix in~(\ref{eq:rcd}) with $d = c^{-1}$:
$$ R_{c,q}(z) := \begin{pmatrix} c-z & & &  \\ & c q^{-1} - z q & (q^{-1} - q)(c^2-1) & \\ & -z & 1 - zc & \\ & & & q^{-1} - z c q \end{pmatrix}. $$
Our work from the previous subsection ensures that $R_{c,q}$
satisfies the following (graded) parametrized Yang-Baxter equation for any $c \in \mathbb{C}^\times$:
\begin{equation} R(z,w)_{12} R_{c,q}(z)_{13} R_{c,q}(w)_{23} = R_{c,q}(w)_{23} R_{c,q}(z)_{13} R(z,w)_{12}, \label{paramybeinc} \end{equation}
with $R(z,w)$ as in~(\ref{eq:rzw}). 
We now create two new $R$-matrices ${R}^\Gamma_{c,q}(z)$ and ${R}^\Delta_{c,q}(z)$. For ${R}^\Gamma_{c,q}(z)$, we perform a linear change of variables, moving an overall factor of $c^2-1$ from the weight in entry $(2,3)$ to the weight in entry $(3,2)$,
thus transforming $R_{c,q}(z)$ to the matrix
$$ {R}^\Gamma_{c,q}(z) := \begin{pmatrix} c-z & & &  \\ & c q^{-1} - z q & q^{-1} - q & \\ & (1-c^2)z & 1 - zc & \\ & & & q^{-1} - z c q \end{pmatrix}. $$
For ${R}^\Delta_{c,q}(z)$, we make a change of variables to move $z$ from weight in position $(3,2)$ to position $(2,3)$ and a factor of $(c+1)$ from position $(2,3)$ to $(3,2)$, and then a very simple Drinfeld twist to introduce a $-1$ to the weights in positions $(2,2)$ and $(3,3)$, resulting in:
$$ {R}^\Delta_{c,q}(z) := \begin{pmatrix} c-z & & &  \\ & - c q^{-1} + z q & (q^{-1}-q)(c-1)z & \\ & -(c+1) & zc-1 & \\ & & & q^{-1} - z c q \end{pmatrix}. $$
Of course, ${R}^\Gamma_{c,q}(z)$ and ${R}^\Delta_{c,q}(z)$ continue to satisfy a Yang-Baxter equation exactly as in \eqref{paramybeinc}.
Moreover, because the entries of the matrices are polynomial in $c$, we obtain additional solutions for certain limits with respect to $c$. In particular, these identities are preserved under the two limits 
$$ A_q(z) := \lim_{c \rightarrow 0} {R}^\Gamma_{c,q}(z), \quad B_q(z) := \lim_{c \rightarrow \infty} \frac{1}{c} {R}^\Delta_{c,q}(z). $$
The next result justifies our naming of the above $R$-matrices.

\begin{proposition} The limiting matrix $(zq) A_q((zq)^{-1})$ equals $T_\Gamma(v,z)$ as given the first row of Table~\ref{tab:boltzmann} with $q = \sqrt{v}$. The limiting matrix $B_q(-qz)$ equals $T_\Delta(v,z)$ as given in the second row of Table~\ref{tab:boltzmann}, again with $q=\sqrt{v}$.
\end{proposition}

\begin{proof}
One readily checks that
$$ zq \cdot \lim_{c \rightarrow 0} {R}^\Gamma_{c,q}((zq)^{-1}) = \begin{pmatrix} -1 & & &  \\ &  - q & (1-q^{2})z & \\ & 1 & qz & \\ & & & z \end{pmatrix}. $$
We change the sign of the entry in position $(1,1)$ to reflect the ungraded solution to the Yang-Baxter equation, and the result matches the weights in Table~\ref{tab:boltzmann} with $q = \sqrt{v}$. Here, to perform the matching with the table, note that Boltzmann weights from Table~~\ref{tab:boltzmann} are arranged in the matrix as follows:
$$ \begin{pmatrix} a_1 & & & \\ & b_1 & c_1 & \\ & c_2 & b_2 & \\ & & & a_2 \end{pmatrix}. $$
Similarly, we match row two of Table~~\ref{tab:boltzmann} with
$$ \lim_{c \rightarrow \infty} \frac{1}{c} {R}^\Delta_{c,q}(-qz) = \begin{pmatrix} 1 & & &  \\ & - q^{-1} & -(1 - q^2)z & \\ & -1 & -qz & \\ & & & q^2 z \end{pmatrix}, $$
upon negating all entries {\it except for} the one in position $(1,1)$ to obtain an ungraded solution, again with $q = \sqrt{v}$.
\end{proof}

The $R$-matrices that result from the limits in the above proposition should arise from superalgebra analogues of the asymptotic 
modules of Hernandez and Jimbo \cite{JimboHernandez}, where in the setting of Kac-Moody Lie algebras $\mathfrak{g}$ these 
modules may be naturally viewed as $U_q(\mathfrak{b})$-modules for a choice of Borel subalgebra $\mathfrak{b}$. This allows one 
to construct limits of certain families of Kirillov-Reshetikhin modules which remain in category $\mathcal{O}$ of $U_q(\mathfrak{b})$-modules.
The superalgebra version of these asymptotic modules is described in~\cite{ZhangSuperAsymptotic} where the role of the Borel subalgebra is
played by the $q$-Yangian $Y_q(\hat{\mathfrak{g}})$ for an affine superalgebra $\hat{\mathfrak{g}}$.
To emphasize our point of view above, we are not taking the limit of modules at the algebra level, but rather concluding the existence of 
Yang-Baxter equations for the limiting $R$-matrices by continuity.

Modulo this detail, we have thus provided an algebraic explanation via affine quantum group modules for the $R$- and 
$T$-matrices of $RTT$ equations appearing in Table~\ref{tab:boltzmann} and the first two lines of Table~\ref{tab:R-vertices}. 
The latter two lines in Table~\ref{tab:R-vertices} have yet to be explained, as we do not have a consistent choice of Hopf algebra basis 
in order to produce $R$-matrices for the $\Gamma$ and $\Delta$ types simultaneously. Indeed, the above proposition requires different 
bases and different Drinfeld twists of the algebra $U_q$ in order to match $T_\Gamma(v,z)$ and $T_\Delta(v,z)$, respectively, so we have not yet 
provided a possible algebraic explanation for these ``mixed'' $R$-matrices $R_{\Gamma \Delta}$ and $R_{\Delta \Gamma}$.
It would be interesting to explore whether they arise naturally from quantum double constructions in the superalgebra setting, where
 a pair of dual Yangians produce Lie superalgebras, leading to a so-called Yang-Baxter system as in 
 \cite{HlavatyYBS, HlavatyNonultralocal, VladimirovDoubles} and discussed in the context of the models of the 
 present paper in Section~9 of the arXiv version of \cite{BrubakerBumpFriedbergHkice}.

\bibliographystyle{habbrv}
\bibliography{tokuyama}

\end{document}